\documentclass[11pt,reqno]{amsart}
\usepackage{amssymb}
\usepackage{amsmath}
\usepackage{color}
\usepackage{graphics, cite, bookmark}
\usepackage{epsfig}
\usepackage{hyperref}
\usepackage{epstopdf}
\newtheorem{theorem}{Theorem}[section]
\newtheorem*{theorem*}{Theorem}
\newtheorem{lemma}{Lemma}[section]
\newtheorem{condition}{Condition}[section]
\newtheorem{proposition}{Proposition}[section]

\theoremstyle{definition}
\newtheorem{definition}{Definition}[section]

\theoremstyle{remark}
\newtheorem{remark}{Remark}[section]

\numberwithin{equation}{section}
\allowdisplaybreaks

\def\R{\mathbb R}
\def\N{\mathbb N}
\def\eps{\varepsilon}
\def\supp{\textrm{supp}}

\begin{document}

\title[Isometric Extension]
{Rigidity and Flexibility of Isometric Extensions}

\author[W. Cao \and D. Inauen]{Wentao Cao \and Dominik Inauen}

\address{Wentao Cao, Academy for Multidisciplinary Studies, Capital Normal University, West 3rd Ring North Road 105, Beijing, 100048 P.R. China. E-mail:{\tt cwtmath@cnu.edu.cn}}
\address{Dominik Inauen, Institut f\"{u}r Mathematik, Universit\"{a}t Leipzig, D-04109, Leipzig, Germany.  E-mail:{\tt dominik.inauen@math.uni-leipzig.de}}

\begin{abstract}
 In this paper we consider the rigidity and flexibility of $C^{1, \theta}$  isometric extensions. We show that the H\"older exponent $\theta_0=\frac12$ is critical in the following sense: if $u\in C^{1,\theta}$ is an isometric extension of a smooth isometric embedding of a codimension one submanifold $\Sigma$ and $\theta> \frac12$, then the tangential connection agrees with the Levi-Civita connection along $\Sigma$. On the other hand, for any
 $\theta<\frac12$  we can construct  $C^{1,\theta}$ isometric extensions via convex integration
 which violate such property. As a byproduct we get moreover an existence theorem for  $C^{1, \theta}$ isometric embeddings, $\theta<\frac12$, of compact Riemannian manifolds with $C^1$ metrics and sharper amount of codimension.
\end{abstract}

\date{\today}

\maketitle
\section{Introduction}
\label{intro}
 Let $(\mathcal M,g)$ be an $n$-dimensional compact smooth Riemannian manifold and $m>n$. Recall that an isometric embedding of $(\mathcal M,g)$ into $(\R^{m},e)$ is an injective $C^{1}$ map $u:\mathcal M\hookrightarrow  \R^{m}$ such that $u^{\sharp}e =g$. Here, $e$ is the Euclidean metric and $u^{\sharp}e$ denotes the pullback metric on $\mathcal M$.  In local coordinates this amounts to the system of  partial differential equations
 \begin{equation}\label{e:isometric}
 \sum_{k=1}^{m} \frac{\partial u^{k} }{\partial x^{i}}\frac{\partial u^{k}}{\partial x^{j}} = g_{ij}
 \end{equation}
 for $1\leq i,j\leq n$, where  $g= g_{ij} dx^{i}dx^{j}$  using summation over repeated indices. 
 
Classical results in differential geometry indicate that \emph{sufficiently regular} global isometric embeddings into Euclidean space with low co-dimension (i.e., $m-n$ is small) are often rigid (i.e., \emph{unique} upto rigid motions). Most prominent is the rigidity of the Weyl problem: given a metric $g$ on the sphere $\mathbb{S}^{2}$ with positive Gaussian curvature, isometric embeddings $u:(\mathbb{S}^{2},g) \hookrightarrow \R^{3}$ are rigid  in the class $C^{2}$ (cf. \cite{CohnVossen,Herglotz}). On the other hand, the celebrated Nash--Kuiper theorem (cf. \cite{Nash54, Kui55}) implies that such spheres can be isometrically embedded into arbitrarily small balls of $\R^{3}$ if one only requires the embedding to be $C^{1}$. A natural question is whether there exists a threshold regularity which distinguishes these two drastically different behaviours.

As shown in \cite{Borisov58,BorisovRigidity1, PogorelovRigidity,Bor04} (see also \cite{CDS12} for a short, modern proof), isometric embeddings $u\in C^{1,\theta}$ of positively curved closed surfaces into $\R^{3}$ are rigid for $\theta>\frac23$. On the other hand, the flexibility of isometric embeddings granted by the Nash--Kuiper theorem also holds for isometric embeddings $u:(\mathcal M,g)\hookrightarrow \R^{n+1}$  of compact $n$-dimensional manifolds  whenever $ u\in C^{1,\theta}$ with $\theta< \frac{1}{1+n+n^{2}}$ for $n\geq 3$ (cf. \cite{Bor65,CDS12,CS20}), and with $\theta <\frac15$ for $n=2$ (cf. \cite{DIS,CS20}). The threshold exponent has been conjectured to be $\theta =\frac12$ (see \cite{Gromov:2015tua}).

The situation looks different when the co-dimension of the ambient space is sufficiently large. A result  by A. K\"all\'en (cf. \cite{K78}) shows that, in this case, flexibility of isometric embeddings extends to the regularity $C^{1,\theta}$ for any $\theta<1$, and thus there is no rigidity below $C^{2}$.

On the other hand, in \cite{DIS20} the authors find that a weaker form of rigidity is present above the conjectured threshold regularity $C^{1,\frac12}$ no matter the codimension: they show that when $u\in C^{1,\theta}( \mathcal M,\R^{m})$ is an isometric embedding with $\theta>\frac12$, a weak notion of tangential connection can be defined on the (irregular) embedded submanifold (see also \cite{Borisov58} for a similar weak notion of tangential connection) and that it agrees with the Levi-Civita connection. In the case of isometric embeddings $u\in C^{1,\theta}((\bar D_1,g),\R^{m})$ of the closed unit disc taking fixed (smooth) boundary values it is then shown that this compatibility of the weak tangential connection with the intrinsic metric leads to an angle constraint of the tangent space at points of the embedded boundary curve. In contrast, for every $\theta<\frac{1}{2}$ the authors construct isometric embeddings $u\in C^{1,\theta}$ violating this constraint by extending the (smooth) boundary datum to a $C^{1,\theta}$ isometric embedding of the disc by means of a convex integration process. Thus, for this particular example, the result in \cite{DIS20} gives a geometric illustration of the criticality of the H\"older exponent $\theta =\frac12$, at least in the presence of a "boundary condition".

In this paper we study the rigidity and flexibility properties of general isometric extensions. A first observation shows that the angle constraint of \cite{DIS20} is simply a consequence of the compatibility of the weak tangential connection with the intrinsic metric and of the embedding agreeing with a smooth one along a lowerdimensional submanifold. It is therefore also present for general isometric extensions of $C^{2}$ isometries which are $C^{1,\theta}$ for $\theta>\frac{1}{2}$.
On the flexibility side, we want to construct isometric extensions $u\in C^{1,\theta}$, $\theta<\frac{1}{2}$, violating this constraint. 

The problem of extending an isometric map $f:\Sigma\to\R^{m}$, where $\Sigma\subset \mathcal M$ is a co-dimension one submanifold, was first considered by Jacobowitz in \cite{Jac74} in the high-regularity  and high co-dimension  setting. He gave a necessary condition on the second fundamental forms of $\iota :\Sigma \hookrightarrow \mathcal M$ and $f:\Sigma\to \R^{m}$ for the existence of a $C^{2}$ extension $u:\mathcal M\hookrightarrow \R^{m}$. He also found a sufficient condition  (which turns out to be almost necessary) for such an extension, which can be stated as that the image $f(\Sigma)$ shall be   ``more curved" than $\Sigma$. Besides, as discussed in \cite{Jac74,CS19}, isometric extension can also be viewed as a Cauchy problem for   isometric embeddings and certain non-degeneracy conditions on the curvature are important to prove the existence of a sufficiently smooth solution (for local extensions from a point resp. a curve on  2-dimensional manifolds, see \cite{Han05,Lin}  resp. \cite{Dong, Kh, Han06, Cao}). The existence of isometric $C^{1}$ extensions in low co-dimension was then investigated in \cite{HunWas16}. The authors showed that Jacobowitz' obstruction for $C^{2}$ extensions is also an obstruction for $C^{1}$ extensions and found a sufficient condition (similar to the one in \cite{Jac74}) for \emph{one-sided} extensions (see \cite{HunWas16} or below for a similar definition). Under such a condition they proved an existence theorem for isometric $C^{1}$ extensions analogous to the Nash--Kuiper theorem. This was then improved to the $C^{1,\theta}$ category for $\theta<\frac{1}{1+n+n^{2}}$ in \cite{CS19}, although the one-sided extensions are only defined locally around a point.

In this paper we show that, under the same sufficient condition, we can find one-sided isometric extensions (defined on a full one-sided neighborhood of the submanifold $\Sigma$) with regularity $C^{1,\theta}$ for $\theta<\frac{1}{2}$, for which the tangential connection does not agree with the Levi-Civita connection along the submanifold $\Sigma$.

To precisely state our results we introduce our setting. We consider a smooth, compact, orientable $n$-dimensional manifold $\mathcal M$ equipped with a $C^{1}$ Riemannian metric $g$ and an orientable submanifold $\Sigma \subset \mathcal M$ of co-dimension one. Suppose that $f:\Sigma \to \R^{m}$ is a smooth isometric embedding for some $m>n$  and denote by $L: T \Sigma\times T\Sigma  \to C^{\infty}(\Sigma)$ the (scalar) second fundamental form of the embedding $\iota: \Sigma \hookrightarrow \mathcal M$, and  by $\bar L:T \Sigma \times T\Sigma \to f^{*}N  f(\Sigma) $ the second fundamental form of the embedding $f$.
In \cite{HunWas16} Hungerb\"uhler--Wasem showed that a sufficient condition for the existence of a $C^1$ one-sided isometric extension (cf. \cite{HunWas16} for the definition) of $f:\Sigma\to \R^{m}$ is that there exists a smooth vectorfield $\mu:\Sigma \to \R^{m}$ satisfying for every $p\in \Sigma$
\begin{equation}\label{e:HW}
\begin{split}
\textrm{(i)}&\quad \mu(p)\in N_{f(p)}f(\Sigma),\\
\textrm{(ii)}&\quad |\mu(p)|=1,\\
\textrm{(iii)}&\quad \langle\mu(p), \bar{L}(\cdot, \cdot)\rangle -L(\cdot,\cdot)\textrm{ is positive definite on }T_p\Sigma.
\end{split}
\end{equation}
Here, $\langle \cdot,\cdot\rangle$ denotes the Euclidean scalar product.
Under this assumption we will be able to extend the isometric embedding $f$ to some neighborhood of $\Sigma$ which is best described by the exponential map.  Let $\nu \in\Gamma(N\Sigma)$  be the unique unit normal vectorfield respecting the orientation of $\Sigma$ in $\mathcal M$. Since $\Sigma$ is compact there exists $\epsilon_0>0$ such that $F:\Sigma\times ]-\epsilon_0,\epsilon_0[\to \mathcal M$ given by
\begin{equation}\label{d:F} F(p,t) = \exp_p(t\nu(p))
\end{equation}
is a diffeomorphism. For $\epsilon \leq \epsilon_0$ we then call $\Sigma_\epsilon^{+} := F(\Sigma\times  [0,\epsilon[)$ a \emph{one-sided neighborhood} of $\Sigma$ in $\mathcal M$.  Lastly, we let
$$\mathcal I^{\theta}_m(\Sigma_\epsilon^{+})=\{v\in C^{1, \theta}|v:(\Sigma_{\epsilon}^{+},g) \hookrightarrow \R^{m} \textrm{ is an isometric embedding with } v|_\Sigma = f\}.$$

Now we are in a position to state our results. One of our main results concerns the rigidity and flexibility of  $C^{1,\theta}$ isometric extensions.

\begin{theorem}\label{t:rigidity-flexibility}
 Let $\Sigma$ be a codimension one oriented submanifold of the compact Riemannian manifold $(\mathcal{M}, g)$, where $ g\in C^1$, and let $\nu$ be the unique unit normal vectorfield respecting the orientation. Suppose moreover that $f:\Sigma\to \R^{m}$ is an isometric embedding satisfying \eqref{e:HW} and let $X\in \Gamma(T \Sigma)$ be any vectorfield tangent to $\Sigma$. Then the following holds.
\begin{enumerate}
\item If $\theta>\frac12,$ $m\geq n+1,$ and  $u\in\mathcal{I}^\theta_m(\Sigma_\epsilon^{+}),$ then $\langle du(\nu), \bar L(X,X)\rangle=L(X, X).$
\item For any $\theta<\frac12,$ $m\geq n+2n_*$, there is $\epsilon>0$ and  $u\in \mathcal{I}^\theta_m(\Sigma_\epsilon ^{+})$ such that $\langle du(\nu), \bar L(X, X)\rangle>L(X, X)$ at all points where $X$ does not vanish. 
\end{enumerate}
\end{theorem}

Here, $n_* =\frac{n(n+1)}{2}$ is the number of equations in \eqref{e:isometric}.  As mentioned above, the proof of part (1) is a simple observation given the result (Proposition 2.2) in  \cite{DIS20} (see Section \ref{s:rigidity}). The isometries in (2) are constructed via a convex integration process similar to \cite{DIS20}. Roughly speaking, the technical difference of (2) to the corresponding part in \cite{DIS20} is that it is of global instead of local nature. Therefore we adapt the "gluing" method introduced in \cite{CS20} for the construction of \emph{global} $C^{1,\theta}$ isometric embeddings to our needs. In particular, to get the regularity $C^{1,\theta}$, $\theta<\frac{1}{2}$, we need a more subtle decomposition lemma (see Lemma \ref{l:perturb}) than in \cite{CS20}, it is similar to the one used in \cite{K78,DIS20}. This however leads to technical difficulties due to the different type of cut-off functions used in \cite{CS20} as compared to \cite{DIS20}; they are resolved in Proposition \ref{p:stage}.

The iteration technique used in the proof of part (2) has its origin in Nash's original construction \cite{Nash54}. The latter inspired the more general framework of convex integration, which remarkably also found application in the question of non-uniqueness of fluid mechanic equations and led for example to the resolution of Onsager's conjecture (see \cite{CET, BDSV17,Ise16}), a striking analogue to the dichotomy of rigidity vs flexibility of isometric embeddings.

One of the main building blocks of the proof of Theorem \ref{t:rigidity-flexibility} (2) is the iteration Proposition \ref{p:inductive}. With it we can prove our second result, the existence of global  isometric embeddings of compact manifolds with $C^1$ metric.

\begin{theorem}\label{t:global}
Let $(\mathcal{M}, g)$ be an $n$-dimensional compact manifold, $n\geq2$, with $C^1$ metric $g$. For any $\theta<\frac12,$ there exist infinitely many $C^{1, \theta}$ isometric embeddings $u:(\mathcal{M}, g)\hookrightarrow \R^{n+2n_*}.$
\end{theorem}

Such an existence result is not new: it is already contained in \cite{K78}. The novelty of Theorem \ref{t:global} is that the target dimension $n+2n_*=n(n+2),$ is much smaller than the dimension $2n+3(n+1)(n^2+n+2)$ in \cite{K78} for the case where the metric is of class $C^1$, i.e. $\beta=1$ in \cite{K78}.

The remaining part of the paper is organized as follows. We prove  Theorem \ref{t:rigidity-flexibility}~(1) in Section \ref{s:rigidity}. The main part of the paper is then devoted to proving Theorem \ref{t:rigidity-flexibility}~(2). We first introduce some notations and useful lemmas in Section \ref{s:preliminary}. We then prove the most important building block, an iteration proposition, in Section \ref{s:iteration}. With it we are able to show Theorem \ref{t:rigidity-flexibility}~(2) and Theorem \ref{t:global} in Section \ref{s:flexibility} and Section \ref{s:global} respectively.

%The remaining part of the paper is organized as follows. We give some notations and useful lemmas in Section \ref{s:preliminary}. Then we give the proof of Theorem \ref{t:rigidity-flexibility}~(1) in Section \ref{s:rigidity}. To prove Theorem \ref{t:rigidity-flexibility}~(2), we first prove an iteration proposition in Section \ref{s:iteration}. With Proposition \ref{p:inductive}, we are able to show Theorem Theorem \ref{t:rigidity-flexibility}~(2) and Theorem \ref{t:global} in Section \ref{s:flexibility} and Section \ref{s:global} respectively.

\section*{acknowledgement}
The authors would like to thank the anonymous referees for the constructive suggestions.
%%%%%%%%%%%%%%%
\section{The proof of Theorem \ref{t:rigidity-flexibility}~(1): Rigidity part}\label{s:rigidity}
%%%%%%%%%%%%%%%%

Recall that for a smooth  isometric  embedding $u:\mathcal M\hookrightarrow \R^{m}$, a curve $\gamma:[0,1]\to \mathcal M$ and a vectorfield $X \in  \Gamma(T\mathcal M)$ it holds by Gauss' formula
\[  \left (\frac{d}{dt}\Big\vert_{t=t_0}du(X)(\gamma)\right )^{\intercal} = du\left (\nabla ^{\mathcal M}_{\dot{\gamma}(t_0)} X\right )\,\]
for $t_0\in [0,1]$. Here, $^{\intercal}$ denotes the orthogonal projection onto $Tu(\mathcal M)\subset \R^{m}$.

Thus, in particular, if $\nu \in \Gamma( T\mathcal M)$, then
\begin{equation}\label{e:projection}
\langle \frac{d}{dt} du(X)(\gamma), du(\nu) (\gamma)\rangle = \langle du\left (\nabla ^{\mathcal M}_{\dot{\gamma}(t)} X\right ), du(\nu) (\gamma)\rangle \,
\end{equation}
holds on $[0,1]$. In Proposition 2.2 of \cite{DIS20}, the authors show that the lefthandside of the latter equation is welldefined as a distribution whenever $u\in C^{1,\theta}$ for $\theta >1/2$, and that \eqref{e:projection} still holds.

From this, part (1) of Theorem \ref{t:rigidity-flexibility} follows easily. Let $X\in \Gamma(T\Sigma)$, $p \in \Sigma$ and $\gamma:[0,1]\to \Sigma$ with $\gamma(0) =p$, $\dot{\gamma}(0) = X_p$. Let moreover $\nu \in \Gamma(T\mathcal M)$ be the unique unit normal  to $\Sigma$ respecting the orientation. Since $u=f$ on $\Sigma$ and $f$ is smooth, the function $du(X)(\gamma): [0,1]\to \R^{m}$ is smooth, so that \eqref{e:projection} holds as a pointwise equality of continuous functions even if $u$ is only $C^{1,\theta}$ with $\theta >\frac{1}{2}$.

With the help of the Gauss formula for the smooth embeddings $f:\Sigma \to \R^{m}$, $\iota: \Sigma \hookrightarrow \mathcal M$  we then find
\begin{align*}
 \langle du_p(\nu) , \bar L_p(X,X) \rangle &= \langle du_p(\nu), \frac{d}{dt}\Big \vert_{t=0} du(X)(\gamma)\rangle = \langle du_p( \nu), du\left (\nabla^{\mathcal M}_{\dot{\gamma}(0) }X\right ) \rangle \\
 &= g\left (\nu_p,\nabla^{\mathcal M}_{X_p} X\right ) =  L_p(X,X)
\end{align*}
since $u$ is an isometry.

\section{Preliminaries}\label{s:preliminary}
%%%%%%%%%%%%%%%%%%
In this section we introduce some notations, function spaces and basic lemmas which are needed for the proof of the flexibility part of Theorem \ref{t:rigidity-flexibility}.

\subsection{H\"older Norms and interpolation} Let $\Omega \subset \R^{n}$ be an open set. In the following, the maps $f$ can be real valued, vector valued, or tensor valued. In every case, the target is equipped with the Euclidean norm, denoted by $|f(x)|$. The H\"older norms are then defined as follows:
\begin{equation*}
\|f\|_0=\sup_{\Omega}|f|, ~~\|f\|_m=\sum_{j=0}^m\max_{|\beta|=j}\|\partial^\beta f\|_0,,
\end{equation*}
where $\beta$ denotes a multi-index,
and
\begin{equation*}
[f]_{\theta}=\sup_{x\neq y}\frac{|f(x)-f(y)|}{|x-y|^\theta},~~[f]_{m+\theta}=\max_{|\beta|=m}\sup_{x\neq y}\frac{ |\partial^\beta f(x)-\partial^\beta f(y)|}{|x-y|^\theta}, 0<\theta\leq1.
\end{equation*}
Then the H\"older norms are given as
$$\|f\|_{m+\theta}=\|f\|_m+[f]_{m+\theta}. $$
We recall the standard interpolation inequality
\begin{equation*}
[f]_r\leq C\|f\|_0^{1-\frac{r}{s}}[f]_s^{\frac{r}{s}}
\end{equation*}
for $s>r\geq0$ and the Leibniz rule
\begin{equation}\label{e:interpolation}
\|fg\|_r\leq C(r)(\|f\|_r\|g\|_0+\|f\|_0\|g\|_r).
\end{equation}
We also collect  two classical estimates on H\"older norms of compositions in \cite{K78}.

\begin{lemma}\label{l:composition}
Let $\Psi: \Sigma_1\to\R$ be a function  and $u, v:\R^n\to\Sigma_1$. Then for any $r, s\geq0$, it holds
\begin{align*}
&\|\Psi\circ u\|_r\leq C(r)(\|\Psi(\cdot)\|_r\|u\|_1^r+\|\Psi(\cdot)\|_1\|u\|_r+\|\Psi(\cdot)\|_0),\, r\geq1,\\
&\|\Psi\circ u\|_r\leq \min(\|\Psi(\cdot)\|_r\|u\|_1^r,\, \|\Psi(\cdot)\|_1\|u\|_r)+\|\Psi(\cdot)\|_0,\, 0\leq r\leq1.
\end{align*}
\end{lemma}
\noindent Other properties of the H\"older norm can be found in references such as \cite{CDS12, DIS}.

\subsection{Mollification estimates}

We will frequently regularize maps by convolution with a standard mollifier, i.e., a radially symmetric $\varphi_\ell\in C^{\infty}_c(B_\ell(0))$ with $\int\varphi_\ell = 1$, where $\ell>0$ denotes the length-scale. Such a regularization of H\"older functions enjoys the following estimates (for a proof, see \cite{CDS12},\cite{DIS}).
\begin{lemma}\label{l:mollification}
For any $r, s\geq0,$ and $0<\theta\leq1,$ we have
\begin{align*}
&[f*\varphi_\ell]_{r+s}\leq C\ell^{-s}[f]_r,\\
&\|f-f*\varphi_\ell\|_r\leq C\ell^{1-r}[f]_1, \text{ if } 0\leq r\leq1,\\
&\|(fg)*\varphi_\ell-(f*\varphi_\ell)(g*\varphi_\ell)\|_r\leq C\ell^{2\theta-r}\|f\|_\theta\|g\|_\theta,
\end{align*}
with constant $C$ depending only on $s, r, \theta, \varphi.$
\end{lemma}
In the course of the proof of Theorem 1.1 we will regularize maps $f\in C^{k}(\bar \Omega,\R^{m})$ defined on $\bar \Omega$. The resulting  regularized maps will then have a smaller domain of definition. To counteract this, we first extend $f$ to a map $\bar f\in C^{k}\left (\R^{n},\R^{m}\right )$ such that
\[ \|f\|_{C^{k}(\R^{m})} \leq C\|f\|_{C^{k}(\bar \Omega)}\,,\]
where the constant $C>0$ depends only on $k,n$ and $\Omega$. Such a procedure is given by Whitney's extension theorem, see \cite{Whitney}. We then mollify the resulting extensions at some length-scale $\ell>0$ to obtain $\tilde f= \bar f\ast \varphi_\ell \in C^{\infty}(\bar \Omega,\R^{m})$. We will not further specify this.

\subsection{H\"older norms and mollification on manifolds}
Using a partition of unity, H\"older spaces and mollification can be defined on the compact manifold $\mathcal M$ as follows. We fix a finite atlas of $\mathcal M$ with charts $(\Omega_i,\phi_i)$, we let $\{\chi_i\}$ be a partition of unity subordinate to $\{\Omega_i\}$ and set
\[ \|f\|_k = \sum_i \|(\chi_i f) \circ \phi_i^{-1}\|_k\,,\]
and
\[ f\ast \varphi_\ell = \sum_i \left ((\chi_i f)\circ \phi_i^{-1}\ast \varphi_\ell\right )\circ \phi_i\,.\]
One can check that the estimates \eqref{e:interpolation} as well as Lemma \ref{l:composition} and \ref{l:mollification} still hold (with constants which may depend on the fixed charts).

\subsection{Matrix decomposition} A key step in the construction of isometric embedding, as pioneered by Nash \cite{Nash54} and used in all the subsequent variants of the iteration process, is a suitable decomposition of the metric error.  We recall the version used in \cite{CDS12, DIS}.

\begin{lemma}\label{l:deco}
Let $n\geq 2$ and let $\bar P \in \R^{n\times n} $ be a positive definite matrix.  There exists a constant $r_0>0$, vectors $\nu_1,\dots,\nu_{n_*}\in \mathbb{S}^{n-1}$ and smooth functions $a_k$ such that

$$
P=\sum_{k=1}^{n_*}a_k^2(P)\nu_k\otimes\nu_k\,,
$$
for any positive definite matrix $P\in \R^{n\times n}$ with
\begin{equation}\label{e:deco}
|P-\bar P|< r_0\,.
\end{equation}

\end{lemma}

\noindent For our purposes we need to perturb Lemma \ref{l:deco} in two ways: first of all we want to vary the ``reference" matrix $\bar P$ slightly and allow a matrixfield $P_0$ with suitably small oscillation; this simply follows from a compactness argument. Secondly, we can perturb the coefficients $a_k$ to obtain a slightly subtler decomposition. This is similar to the decomposition used in \cite{K78} and can proved with the standard implicit function theorem (compare Proposition 5.4 in \cite{DIS20}).

\begin{lemma}\label{l:perturb}
Let $n\geq 2$ and  $\gamma \geq 1$. There exists a constant $\sigma_0(\gamma)>0$ and vectors $\nu_1,\ldots, \nu_{n*} \in \mathbb S^{n-1}$ with the following property.
If $P_0:\bar \Omega \to \R^{n\times n}$  is a matrix field with $\gamma^{-1}\mathrm{Id}\leq P_0\leq \gamma \mathrm{Id}$ and $\mathrm{osc}_\Omega P_0 <\sigma_0$, and if $P:\bar\Omega\to\R^{n\times n}_{sym}$, and $\{\Lambda_i\}_{i=1}^{n_*},$ $\{\Theta_{ij}\}_{i, j=1}^{n_*}\subset C^1(\bar\Omega, \R^{n\times n}_{sym})$ are such that
 \begin{equation}\label{e:peturb-mn}
 \|P-P_0\|_0+\sum_{ i=1}^{n_*}\|\Lambda_i\|_0+\sum_{i, j=1}^{n_*}\|\Theta_{ij}\|_0< \sigma_0,
 \end{equation}
 then there exist $C^1$ functions $a_1, \cdots, a_{n_*}:\bar\Omega\to\R$ with
\begin{equation}\label{e:decomp-p}
P=\sum_{i=1}^{n_*}a_i^2\nu_i\otimes\nu_i+\sum_{i=1}^{n_*}a_i\Lambda_i+\sum_{i, j=1}^{n_*}a_ia_j\Theta_{ij}.
\end{equation}
Moreover, $a_i$ are given as
\begin{equation}\label{e:a_i}a_i(x) = \Phi_i(P(x), \{ \Lambda_k(x)\},\{\Theta_{kl}(x)\})\,\end{equation}
for $C^{1}$ functions $\Phi_i$, and consequently, we have the estimates
\begin{equation*}
\|a_i\|_{k}\leq C_{k}\left(\|P\|_{k}+\sum_{j=1}^{n_*}\|\Lambda_j\|_k+\sum_{j,l=1}^{n_*}\|\Theta_{jl}\|_k\right),
\end{equation*}
for $k=0, 1$ and $1\leq i\leq n_*.$ Here, the constants $C_{k}\geq 1$ depend only on $k, \sigma_0$.
\end{lemma}

\subsection{Existence of normal vectorfields} Another key ingredient in the iteration process is the use of suitable normal vector fields to the embedding. The following lemma concerns their existence. It is similar to Proposition 5.3 in \cite{DIS20}, except that no $C^{1}$-closeness to a reference embedding is required. A proof is contained in the appendix.
\begin{lemma}\label{l:normal}
Let $N\in\N$ and $\Omega\subset\R^n$ open and simply connected. Assume $v\in C^{N+1}(\bar\Omega,\R^{m})$ is such that
\begin{equation} \label{e:vispositivedefinite}\gamma^{-1}\textrm{Id}\leq\nabla v^T\nabla v\leq \gamma \textrm{Id}\end{equation}
for some $\gamma>1.$ There exists a family of vectorfields $\{\zeta_1,\cdots,\zeta_{m-n}\}\subset C^{N}(\bar\Omega, \R^m)$ such that
\begin{align}
&\langle\zeta_i,\zeta_j\rangle=\delta_{ij},\,\, \nabla v\cdot\zeta_i=0,\nonumber\\
&[\zeta_i]_{l}\leq C(l,\gamma)(1+[v]_{l+1}),\label{e:normalestimates}
\end{align}
for all $0\leq l\leq N.$ Here $\delta_{ij}=1$ when $i=j$ and vanishes else.
\end{lemma}

%%%%%%%%%%%%%%%%%
\section{Iteration proposition}\label{s:iteration}
%%%%%%%%%%%%%%%%%%
The isometric embedding $u\in \mathcal I^{\theta}_m(\Sigma_\epsilon^{+})$ in Theorem \ref{t:rigidity-flexibility} (2) will be constructed by an iteration procedure producing a sequence of embeddings $u_q$ converging to $u$. The practice, as pioneered by Nash in \cite{Nash54}, of decomposing the metric error $g-u_q^{\sharp}e$ and adding a Nash-twist for each term in the decomposition was improved by A. K\"allen in \cite{K78}. In the latter paper, the author gains extra regularity (at the expense of increased codimension compared to Nash's result) by absorbing the leading error terms into the decomposition. We use a similar decomposition (see \eqref{e:decomp-p}). However, since we employ the framework of \cite{CS20} we need to be able to "add" metric pieces which have the form $\rho^{2}(g+h)$ for compactly supported functions $\rho$ and suitably small $(0,2)$-tensors $h$. The missing lower bound on $\rho$ seems however not to be compatible with the decomposition lemma; a technical difficulty which is overcome by introducing an extra cut-off scale (see \eqref{d:epsilon} and \eqref{e:cutoffscale}).

In this proposition $G$ is assumed to be the coordinate expression of a given $C^1$ metric in some chart which is identified with an open bounded subset $\Omega\subset\R^n$. Moreover, the constant $\sigma_0$ is given by Lemma  \ref{l:perturb}.

\begin{proposition}\label{p:stage}
Fix $\gamma\geq 1$ and parameters $0<\delta<1$ and $\lambda >1$. Assume $G$ is a $C^1$ metric on $\bar\Omega\subset\R^n$ with ${\gamma}^{-1}\textrm{Id}\leq G\leq \gamma\textrm{Id}$, $\|G\|_1\leq \gamma$ and $\textrm{osc}_{\Omega }G<\sigma_0$, and $u\in C^2(\bar \Omega, \R^{n+2n_*})$, $\rho\in C^1(\bar \Omega)$ and $H\in C^1(\bar\Omega;\R^{n\times n}_{sym})$ are such that
\begin{equation}\label{e:u-assumption}
\begin{split}
{\gamma}^{-1}\textrm{Id}\leq \nabla u^T&\nabla u\leq\gamma \textrm{Id}\,\textrm{ in }\Omega,\\
\|u\|_{2}&\leq \delta^{1/2}\lambda,
\end{split}
\end{equation}
and
\begin{align}
&0\leq \rho\leq \delta^{1/2}, \quad \|\rho\|_1\leq \delta^{1/2}\lambda, \label{e:rho-assumption}\\
&\|H\|_0\leq \frac{\sigma_0}{2}, \quad \|H\|_1\leq \lambda\,.\label{e:H-assumption}
\end{align}
 Then, for every $\tau>1$ there exists a constant $\lambda_0(\tau, \gamma,\sigma_0)\geq 1$  such that if
\begin{equation}\label{e:stageparameters}
 \lambda \geq \lambda_0\,,
\end{equation} then there exists an embedding $v\in C^2(\bar \Omega; \R^{n+2n_*})$ and $\mathcal{E}\in C^1(\bar \Omega;\R^{n\times n}_{sym})$ such that
\begin{equation}\label{e:stage-support}
\begin{split}
 \nabla v^T\nabla v=&\nabla u^T\nabla u+\rho^2(G+H)+\mathcal{E}\quad\textrm{ in }\Omega,\\
v=&u\textrm{ on }\Omega\setminus(\supp\rho +B_{\lambda^{1-2\tau}}(0))
\end{split}
\end{equation}
with estimates
\begin{align}
\|v-u\|_0&\leq C\delta^{1/2}\lambda^{-\tau},\label{e:stage-2}\\
\|v-u\|_1&\leq C\delta^{1/2},\label{e:stage-3}\\
\|v\|_2&\leq C\delta^{1/2}\lambda^{\tau},\label{e:stage-4}
\end{align}
and
\begin{equation}\label{e:stage-5}
\|\mathcal{E}\|_0\leq C\delta\lambda^{2-2\tau}, \quad
\|\mathcal{E}\|_1\leq C\delta\lambda.
\end{equation}
Here, $C\geq 1$ is a constant depending only on $\gamma,\sigma_0$.
\end{proposition}

\begin{remark}[Constants] As usual, the value of the constants $C$ appearing in the following proof can change from line to line. In addition, all the constants are allowed to depend on $\gamma$ and $\sigma_0$. For the sake of readibility, we will suppress this dependence in the notation.
\end{remark}
%\begin{remark}[Regularization] \label{r:regularization} In the course of the proof we will regularize maps $f\in C^{k}(\bar \Omega,\R^{m})$ defined on $\bar \Omega$ by mollification. The resulting  regularized maps will then have a smaller domain of definition. To counteract this, we first extend $f$ to a map $\bar f\in C^{k}\left (\R^{n},\R^{m}\right )$ such that
%\[ \|f\|_{C^{k}(\R^{m})} \leq C\|f\|_{C^{k}(\bar \Omega)}\,,\]
%where the constant $C>0$ depends only on $k,n$ and $\Omega$. Such a procedure is wellknown. We then mollify the resulting extensions at some length-scale $\ell>0$ to obtain $\tilde f= \bar f\ast \varphi_\ell \in C^{\infty}(\bar \Omega,\R^{m})$. We will not further specify this in the proof.
%\end{remark}
\begin{proof}
Fix $\tau> 1$. Regularize  $u$ at length scale $\lambda^{-\tau}$  to get $\tilde u\in C^{\infty}(\bar \Omega)$.
Then the smooth embedding $\tilde u$ satisfies
$$
\|u-\tilde u\|_1\leq C\delta^{1/2}\lambda^{1-\tau},\quad \|\tilde u\|_2\leq C\delta^{1/2}\lambda,
\quad \|\tilde u\|_3\leq C\delta^{1/2}\lambda^{\tau+1}.
$$
Note that
$$
\nabla\tilde u^T\nabla \tilde u=\nabla u^T\nabla u-(\nabla u-\nabla\tilde u)^T\nabla\tilde u
-\nabla u^T(\nabla u-\nabla\tilde u),
$$
which then implies
$$
(2\gamma)^{-1}\textrm{Id}\leq\nabla\tilde u^T\nabla\tilde u\leq (2\gamma)\textrm{Id}.
$$
provided $\lambda^{1-\tau} \leq C(\gamma)^{-1}$ for some constant $C(\gamma)$, which follows from \eqref{e:stageparameters} for $\lambda_0$ large enough.  Then $\tilde u:\bar \Omega\hookrightarrow\R^{n+2n^*} $ is an embedding of $\bar \Omega.$ Thus by Lemma \ref{l:normal}, there exist
 $2n_*$ unit normal vectors $\{\zeta_k, \eta_k, k=1,\cdots, n_*\}$ to the surface $\tilde u(\bar\Omega)$ satisfying the estimates \eqref{e:normalestimates}. Fix moreover the vectors $\nu_1,\ldots, \nu_{n*}\in \mathbb S^{n-1}$ provided by Lemma \ref{l:perturb} for our fixed $\gamma\geq 1$.
Similarly to \cite{DIS20}, we then  define
\begin{align*}
A_k=&\cos(\lambda^\tau\nu_k\cdot x)\zeta_k\otimes\nu_k-\sin(\lambda^\tau\nu_k\cdot x)\eta_k\otimes\nu_k,\\
B_k=&\sin(\lambda^\tau\nu_k\cdot x)\nabla\zeta_k+\cos(\lambda^\tau\nu_k\cdot x)\nabla\eta_k,\\
D_k=&\sin(\lambda^\tau\nu_k\cdot x)\zeta_k+\cos(\lambda^\tau\nu_k\cdot x)\eta_k\,.
\end{align*}
By \eqref{e:interpolation}, it is not hard to derive
\begin{equation}\label{e:a-b-d-k}
\begin{split}
\|A_k\|_0+\|D_k\|_0\leq &C(1+\|\nabla\tilde u\|_0)\leq C,\\
\|A_k\|_1+\|D_k\|_1\leq& C(\lambda^\tau\|\nabla\tilde u\|_0+\|\nabla^2\tilde u\|_0)\leq C(\lambda^\tau+\delta^{1/2}\lambda)\leq C\lambda^\tau,\\
\|B_k\|_0\leq &C\|\nabla^2\tilde u\|_0\leq C\delta^{1/2}\lambda,\\
\|B_k\|_1\leq&C(\|\nabla^2\tilde u\|_0\lambda^\tau+\|\tilde u\|_3)\leq C\delta^{1/2}\lambda^{1+\tau}.
\end{split}
\end{equation}
Note that $\nabla\tilde u^TA_k=0.$ Thus we have
\begin{equation}\label{e:uak}
\begin{split}
\|\nabla u^T A_k\|_0=&\|(\nabla u-\nabla \tilde u)^T A_k\|_0\leq C\delta^{1/2}\lambda^{1-\tau}.\\
\|\nabla u^TA_k\|_1\leq& C(\|\nabla\tilde u\|_1\|A_k\|_0+\|\nabla u-\nabla\tilde u\|_0\|A_k\|_1)\\
\leq &C(\delta^{1/2}\lambda+\delta^{1/2}\lambda^{1-\tau}\lambda^\tau)\leq C\delta^{1/2}\lambda.
\end{split}
\end{equation}
Clearly, we have the same estimates for  $\nabla u^T D_k$:
\begin{equation}\label{e:udk} \|\nabla u^T D_k\|_0 \leq C \delta^{1/2}\lambda^{1-\tau}\,, \|\nabla u^{T}D_k\|_1 \leq C\delta^{1/2} \lambda\,.\end{equation}
 We now set
 \begin{align*}
 \Lambda_k=&2\,\textrm{sym}(\nabla u^T A_k)+2\lambda^{-\tau}\textrm{sym}(\nabla u^TB_k)\,,\\
 \Theta_{ij}=&2\lambda^{-\tau}\textrm{sym}(A_i^TB_j)+2\lambda^{-2\tau}\textrm{sym}(B_i^TB_j).
 \end{align*}
  With the help of \eqref{e:a-b-d-k}--\eqref{e:udk} we then deduce
\begin{equation}\label{e:LT-k-ij}
\begin{split}
\|\Lambda_k\|_0\leq& C(\|\nabla u^TA_k\|_0+\lambda^{-\tau}\|\nabla u\|_0\|B_k\|_0)
\leq C\delta^{1/2}\lambda^{1-\tau}\,,\\
\|\Lambda_k\|_1\leq& C\left (\|\nabla u^TA_k\|_1+\lambda^{-\tau}\left (\|\nabla u\|_1\|B_k\|_0+\|\nabla u\|_0\|B_k\|_1\right ) \right )\leq C\delta^{1/2}\lambda\,,\\
\|\Theta_{ij}\|_0\leq &C\lambda^{-\tau}(\|A_i\|_0\|B_j\|_0+\lambda^{-\tau}\|B_i\|_0\|B_j\|_0)\leq C(\gamma)\delta^{1/2}\lambda^{1-\tau},\\
\|\Theta_{ij}\|_1\leq &C\lambda^{-\tau}(\|A_i\|_1\|B_j\|_0+\|A_i\|_0\|B_j\|_1+\lambda^{-\tau}(\|B_i\|_1\|B_j\|_0+\|B_i\|_0\|B_j\|_1))\\
\leq& C(\delta^{1/2}\lambda+\delta\lambda^{2-\tau})\leq C\delta^{1/2}\lambda.
\end{split}
\end{equation}
Fix now a parameter  $\epsilon >0$ defined by
\begin{equation}\label{d:epsilon}\epsilon^{1/2}= C_0(\gamma, \sigma_0) \delta^{1/2}\lambda^{1-\tau}\,,\end{equation}
where $C_0(\gamma,\sigma_0)\geq 1$ is a constant to be chosen later. Observe that upon choosing $\lambda_0(\gamma,\sigma_0,\tau)$ large enough we can achieve $\epsilon^{1/2}<\delta^{1/2}$.  Next, fix a  monotone decreasing function $\psi\in C^{\infty}([0,\infty[)$ such that
\begin{equation}\label{e:cutoffscale} \psi(\rho) =\begin{cases} \,\,\frac{1}{\rho} \quad \, \text{ if } \rho\geq 2\epsilon^{1/2},\\
\epsilon^{-1/2} \text{ if } \rho \leq \epsilon^{1/2}\,.\end{cases}\end{equation}
Clearly,
\[ \|\psi(\rho(\cdot))\|_0   \leq C\epsilon^{-1/2} \,,\|\psi(\rho(\cdot))\|_1 \leq C\epsilon^{-1}\delta^{1/2}\lambda\,,\]
due to the assumption \eqref{e:rho-assumption}. It therefore follows that
\begin{equation}\label{e:rholambda_k}
\begin{split}
\| \psi(\rho) \Lambda_k \|_0& \leq C \epsilon^{-1/2}\delta^{1/2} \lambda^{1-\tau}\,,\\
\|\psi(\rho) \Lambda_k\|_1 &\leq C\left (\epsilon^{-1/2}\delta^{1/2}\lambda + \epsilon^{-1}\delta \lambda^{2-\tau}\right ) \leq C\epsilon^{-1/2}\delta^{1/2}\lambda \,,
\end{split}
\end{equation}
since $C_0\geq 1$.
Thus, if $C_0$ in \eqref{d:epsilon} is chosen large enough (and afterwards $\lambda_0$ in \eqref{e:stageparameters} is large enough to guarantee $\epsilon <\delta$) we have the following bound
$$\|H\|_0+\sum_{k=1}^{n_*}\|\psi(\rho)\Lambda_k\|_0+\sum_{i, j=1}^{n_*}\|\Theta_{ij}\|_0\leq\frac{\sigma_0}{2}+C\epsilon^{-1/2}\delta^{1/2}\lambda^{1-\tau} < \sigma_0.$$
A direct application of Lemma \ref{l:perturb} (with $P_0 = G$, $P=G+H$) then enables us to get $n_*$ functions $\{a_k\}\subset C^{1}(\bar\Omega)$ such that
\begin{equation*}
G+H=\sum_{k=1}^{n_*}a_k^2\nu_k\otimes\nu_k+\sum_{k=1}^{n_*}a_k\psi(\rho)\Lambda_k+\sum_{i,j=1}^{n_*}a_ia_j\Theta_{ij}\,,
\end{equation*}
i.e.,
\begin{equation}\label{e:rho-g-h-deco}
\rho^{2}(G+H)=\sum_{k=1}^{n_*}(\rho a_k)^2\nu_k\otimes\nu_k+\sum_{k=1}^{n_*}\rho^{2}\psi(\rho) a_k \Lambda_k+\sum_{i,j=1}^{n_*}(\rho a_i) (\rho a_j)\Theta_{ij}\,.
\end{equation}
Notice that $\rho^{2}\psi(\rho) = \rho$ if $\rho \geq 2\epsilon^{1/2}$, so that, at least in this region, we get a decomposition of the form \eqref{e:decomp-p} for the degenerate metric piece $\rho^{2}(G+H)$ with coefficients $\rho a_k$.
By Lemma \ref{l:perturb} we have for $k=1, \cdots, n_*,$
\begin{equation}\label{e:ak-bound}
\begin{split}
0\leq a_k\leq& C\left(\|G\|_0+\|H\|_0+\sum_{k=1}^{n_*}\|\psi(\rho)\Lambda_k\|_0+\sum_{i,j=1}^{n_*}\|\Theta_{ij}\|_0\right)\leq C
,\\
\|a_k\|_1\leq& C\left(\|G\|_1+\|H\|_1+\sum_{k=1}^{n_*}\|\psi(\rho)\Lambda_k\|_1+\sum_{i,j=1}^{n_*}\|\Theta_{ij}\|_1\right)\leq C\epsilon^{-1/2}\delta^{1/2}\lambda.
\end{split}
\end{equation}
However, it follows from the description \eqref{e:a_i} and the estimates \eqref{e:LT-k-ij} that the following improved estimate  holds
\begin{align*}
&\|\rho \nabla a_k \|_0 \leq C\|\rho\nabla(\psi(\rho)\Lambda_k)\|_0
\leq C\left (\|\rho\psi'(\rho) \Lambda_k\nabla\rho  \|_0+\|\rho\psi(\rho)\nabla\Lambda_k\|_0\right )\\
\leq&C(\epsilon^{-1/2}\delta\lambda^{2-\tau}+\delta^{1/2}\lambda)\leq C\delta^{1/2}\lambda\,,
\end{align*}
since $|\rho\psi'(\rho)|\leq C\epsilon^{-1/2}$ and $|\rho\psi(\rho)|\leq C.$
We can then infer that by \eqref{e:interpolation}
\begin{equation}\label{e:ak-bound2}
 \| \rho a_k \|_1 \leq C(\|a_k\nabla\rho\|_0+\|\rho\nabla a_k\|_0)\leq C\delta^{1/2}\lambda \,.
\end{equation}
Now we set $b_k := \rho a_k$ and mollify $b_k$ at length scale  $\lambda^{1-2\tau}$ to get $\tilde b_k$. By \eqref{e:rho-assumption}, \eqref{e:ak-bound2} and Lemma  \ref{l:mollification}, we have for any $j\in \N$
%\begin{align*}
%\|\tilde \rho\|_0\leq \delta^{1/2}, \quad & \|\tilde a_k\|_{0}\leq C(\gamma,\sigma_0, r_0),\\
%\|\tilde \rho\|_{j+1}\leq C(j)\delta^{1/2}\lambda^{2\tau j-j+1},\quad & \|\tilde a_k\|_{j+1}\leq C_j(\gamma,\sigma_0, r_0)\lambda^{2\tau j-j+1},\\
%\|\tilde \rho-\rho\|_0\leq C\delta^{1/2}\lambda^{2-2\tau},\quad  &\|\tilde a_k-a_k\|_0\leq C(\gamma,\sigma_0, r_0)\lambda^{2-2\tau},
%\end{align*}
%for all $j\geq 0$. Hence, we infer
\begin{equation}\label{e:a-rho-c-j-bound}
\begin{split}
\| \tilde b_k \|_0 &\leq  C \delta^{1/2}\\
\|\tilde b_k \|_{j+1} &\leq  C_j \delta^{1/2}\lambda^{(2\tau-1)j +1}\\
\|\tilde b_k -b_k \|_{0} &\leq C_j \delta^{1/2}\lambda^{2-2\tau}\,.
\end{split}
\end{equation}
Finally, we define our desired embedding as
 $$
 v=u+\frac{1}{\lambda^{\tau}}\sum_{k=1}^{n_*}\tilde b_k D_k.
 $$
From the definition it is clear that $v=u$ on $\Omega\setminus \left (\supp \rho + B_{\lambda^{1-2\tau}}\right )$, i.e.,   \eqref{e:stage-support} holds. Besides,
 a straightforward calculation shows
$$\nabla v=\nabla u+\sum_{k=1}^{n_*}\tilde b_k A_k+\frac{1}{\lambda^{\tau}}\sum_{k=1}^{n_*}\tilde b_k B_k+
\frac{1}{\lambda^{\tau}}\sum_{k=1}^{n_*}D_k\nabla \tilde b_k  \,.$$
 Since $A_k^TD_k=0$ for $k=1, 2, \cdots, n_*$,  the induced metric will be
 \begin{align*}
 \nabla v^T\nabla v=&\nabla u^T\nabla u+\sum_{k=1}^{n_*}\tilde b	_k^2\nu_k\otimes\nu_k
 +2\sum_{k=1}^{n_*}\tilde b_k\textrm{sym}\left (\nabla u^{T} A_k +\frac{1}{\lambda^{\tau}}\nabla u^{T}B_k\right ) \\
  & \,+ \frac{2 }{\lambda^{\tau}}\sum_{k=1}^{n_*}\textrm{sym}\left (\nabla u^{T}D_k\nabla\tilde b_k\right )+ \frac{2 }{\lambda^{\tau}}\sum_{i,j=1}^{n_*}\tilde b_i\tilde b_j\mathrm{sym}\left (A_i^{T}B_j\right )
 \\& + \frac{2}{\lambda^{2\tau}}\sum_{i,j=1}^{n_*}\tilde b_i\tilde b_j \mathrm{sym}\left (B_i^{T}B_j\right )
+\frac{2}{\lambda^{2\tau}}\sum_{i,j=1}^{n_*}\tilde b_i \mathrm{sym}\left (B_i^{T}D_j\nabla \tilde b_j\right ) \\
& +\frac{\tilde \chi^{2}}{\lambda^{2\tau}} \sum_{k=1}^{n_*} \nabla \tilde b_k ^{T}\nabla \tilde b_k \,.
 \end{align*}
 Hence by \eqref{e:rho-g-h-deco}, we calculate the metric error
 $$\nabla v^T\nabla v-(\nabla u^T\nabla u+\rho^2(G+H))=\mathcal{E}_1+\mathcal{E}_2,$$
 with
 \begin{align*}
 \mathcal{E}_1=&\sum_{k=1}^{n_*}(\tilde b_k^{2}-b_k^{2})\nu_k\otimes\nu_k
 +\sum_{k=1}^{n_*}(\tilde b_k -\rho\psi(\rho)b_k)\Lambda_k+\sum_{i, j=1}^{n_*}(\tilde b_i\tilde b_j - b_ib_j)\Theta_{ij},\\
 \mathcal{E}_2=& \frac{2}{\lambda^{\tau}}\sum_{k=1}^{n_*}\textrm{sym}\left (\nabla u^{T}D_k\nabla\tilde b_k\right ) + \frac{2}{\lambda^{2\tau}}\sum_{i,j=1}^{n_*}\tilde b_i \mathrm{sym}\left (B_i^{T}D_j\nabla \tilde b_j\right )+\frac{1}{\lambda^{2\tau}} \sum_{k=1}^{n_*} \nabla \tilde b_k ^{T}\nabla \tilde b_k\,.
 \end{align*}
 In the following, we  shall bound the above two errors in order to get \eqref{e:stage-5}. We start with
 \begin{align*}
 \|\tilde b_k^{2}- b_k^{2}\|_0&\leq \| \tilde b_k + b_k\|_0\|\tilde b_k -b_k\|_0 \leq C \delta \lambda^{2-2\tau}\,,
 \end{align*}
 where we used \eqref{e:a-rho-c-j-bound}.  Similarly, with \eqref{e:interpolation} one can then estimate
 \[  \|\tilde b_k^{2}- b_k^{2}\|_1 \leq C \delta\lambda\,.\]
  Completely analogously one can estimate the term
  \[\|(\tilde b_i\tilde b_j - b_ib_j)\Theta_{ij}\|_0\leq  C\delta^{3/2}\lambda^{3-3\tau} \leq C \delta \lambda^{2-2\tau}\]
  and
  \[\|(\tilde b_i\tilde b_j - b_ib_j)\Theta_{ij}\|_1 \leq C\delta\lambda \,.\]
 For the second term in $\mathcal E_1$ we  write
 \[\tilde b_k - \rho\psi(\rho) b_k = \tilde b_k-b_k +b_k(1-\rho\psi(\rho) )\,\]
 and observe that by definition $1-\rho\psi(\rho) = 0$ for $\rho\geq 2\epsilon^{1/2}$. Thus, remembering that $b_k=\rho a_k$, we have that $|b_k|\leq C\epsilon^{1/2}$ on  $\mathrm{spt}(1-\rho\psi(\rho)))$.
 Hence,
  \begin{align*}
  \|\left (\tilde b_k -\rho\psi( \rho) b_k\right )\Lambda_k \|_0 &\leq C\delta^{1/2}\lambda^{1-\tau} \left (\|\tilde b_k-b_k\|_0 + \|b_k(1-\rho\psi(\rho))\|_0 \right )\\
  &\leq C\delta^{1/2}\lambda^{1-\tau}\left (\delta^{1/2}\lambda^{2-2\tau} +\epsilon^{1/2} \right ) \leq C\delta \lambda ^{2-2\tau}
  \end{align*}
  by the definition of $\epsilon$ in \eqref{d:epsilon}.
 Similarly,
 \begin{align*}\|\left (\tilde b_k - \rho\psi(\rho)b_k\right )\Lambda_k \|_1 &\leq C\delta^{3/2}\lambda^{3-2\tau} +C\delta^{1/2}\lambda^{1-\tau}(\delta^{1/2}\lambda + \epsilon^{1/2}\|\nabla(1-\rho\psi(\rho))\|_0 )\\
 &\leq C \delta\lambda^{2-\tau}\leq C\delta\lambda\,, \end{align*}
since $|\nabla(\rho\psi(\rho))|\leq C\epsilon^{-1/2}\delta^{1/2}\lambda+ C|\psi(\rho)\nabla\rho |\leq C\epsilon^{-1/2}\delta^{1/2}\lambda$.

Combining the previous estimates, we get
 \[ \|\mathcal E_1 \|_0 \leq C \delta\lambda^{2-2\tau},\, \|\mathcal E_1\|_1 \leq C\delta \lambda\,.\]
  The estimation of $\mathcal E_2$ is lengthy but straightforwardly obtained by \eqref{e:interpolation}, \eqref{e:a-b-d-k}, \eqref{e:udk}, \eqref{e:a-rho-c-j-bound} and yields
 \[ \|\mathcal E_2\|_0 \leq C\delta \lambda^{2-2\tau},\, \|\mathcal E_2\|_1 \leq C\delta \lambda\,.  \]
This in turn implies \eqref{e:stage-5}, since
\begin{equation*}%\label{e:metric-error1}
\begin{split}
&\|\nabla v^T\nabla v-(\nabla u^T\nabla u+\rho^2(G+H))\|_0\leq\|\mathcal{E}_1\|_0+\|\mathcal{E}_2\|_0 \leq C\delta\lambda^{2-2\tau}\,,\\
&\|\nabla v^T\nabla v-(\nabla u^T\nabla u+\rho^2(G+H))\|_1\leq\|\mathcal{E}_1\|_1+\|\mathcal{E}_2\|_1 \leq  C \delta\lambda\,.
\end{split}
\end{equation*}

\smallskip
It remains to show the estimates \eqref{e:stage-2}-\eqref{e:stage-4}. Clearly, by the formulae for $v$ and its derivative and the estimates \eqref{e:a-b-d-k}, \eqref{e:a-rho-c-j-bound}  we have
\begin{align*}
\|v-u\|_0\leq& \lambda^{-\tau}\sum_{k=1}^{n_*}\|\tilde b_k\|_0\|D_k\|_0\leq C\delta^{1/2}\lambda^{-\tau},
\end{align*}
and
\begin{align*}
\|v-u\|_1\leq&\sum_{k=1}^{n_*}\|\tilde b_k\|_0\| A_k\|_0+\lambda^{-\tau}\sum_{k=1}^{n_*}(\|\tilde b_k\|_0\|B_k\|_0+\|D_k\|_0\|\nabla\tilde b_k\|_0) \\
\leq &C(\delta^{1/2}+\delta^{1/2}\lambda^{1-\tau})\leq C\delta^{1/2}
\end{align*}
Thus we achieve \eqref{e:stage-2}-\eqref{e:stage-3}.
For the second derivatives, we also apply \eqref{e:a-b-d-k}, \eqref{e:a-rho-c-j-bound} and \eqref{e:interpolation} to obtain
\begin{align*}
\|v-u\|_2\leq&  C\sum_{k=1}^{n_*}\left ( \|\tilde b_k\|_0\| A_k+\lambda^{-\tau}B_k\|_1 + \|\tilde b_k\|_1\| A_k+\lambda^{-\tau}B_k\|_0 + \lambda^{-\tau}\|D_k\nabla \tilde b_k\|_1 \right ) \\
\leq & C( \delta^{1/2}\lambda^{\tau} + \delta^{1/2}\lambda + \delta^{1/2}\lambda^{\tau})
\leq C\delta^{1/2}\lambda^{\tau}.
\end{align*}
With \eqref{e:u-assumption} and the fact $\tau>1$, we arrive at \eqref{e:stage-4} and finish the proof.
\end{proof}

%%%%%%%%%

With Proposition \ref{p:stage}, we can modify the inductive Proposition 4.1 in \cite{CS20} to fit our setting, which will help us to construct  \emph{adapted short embeddings} iteratively. We recall the  definition of adapted short embeddings  from \cite{CS19, CS20}.

\begin{definition}\label{d:adapted}
Given a closed subset $\Sigma\subset \mathcal M$ and $\theta\in \,]0,1[$, an embedding  $u:\mathcal M\rightarrow\R^{m}$  is called \emph{adapted short embedding with respect to $\Sigma$ with exponent $\theta$} if
\begin{enumerate}
\item $u\in C^{1, \theta}(\mathcal M)$,
\item there exists a nonnegative function $\rho \in C(\mathcal M)$ with $\Sigma =\{\rho =0 \}$ and a symmetric $(0,2)$-tensor $h\in C(\mathcal M)$ with  $-\frac{1}{2}g \leq h \leq \frac{1}{2}g$  such that
\[ g-u^{\sharp}e = \rho^{2}( g+h)\,, \]
\item $u\in C^{2}(\mathcal M\setminus \Sigma)$, $\rho, h\in C^{1}(\mathcal M\setminus \Sigma)$ and there exists a constant $A\geq 1$ such that in any chart $\Omega_k$
\begin{align}
|\nabla^2u(x)|&\leq A\rho(x)^{1-\frac1\theta}, \label{e:u-adapt}\\
|\nabla\rho(x)|&\leq A\rho(x)^{1-\frac1\theta},\label{e:rho-adapt} \\
\quad |\nabla h(x)|&\leq A\rho(x)^{-\frac1\theta},\label{e:g-adapt}
\end{align}
for any  $x\in \Omega_k\setminus \Sigma$.
\end{enumerate}
\end{definition}

Let $u$ be an adapted short embedding with respect to some compact set $\Sigma\subset \mathcal{M}$ with exponent $\theta$ (c.f.~Definition \ref{d:adapted}). In particular
$$
g-u^\sharp e=\rho^2(g+h),
$$
with $\Sigma=\{\rho=0\}$. Furthermore, let $S\supset\Sigma$ be another compact subset. Our next goal is to show that, under certain conditions, we can perturb $u$  using Proposition \ref{p:stage} to construct another adapted short embedding with respect to the larger compact set $S$ with some exponent $\theta'<\theta$. In particular, we will be able to successively perturb $u$ to make it isometric along the skeleta of a suitable triangulation, eventually ending up with an isometry of a neighborhood of $\Sigma$ for the flexibility part of Theorem \ref{t:rigidity-flexibility}, respectively an isometry of $\mathcal M$ in Theorem \ref{t:global}.
We recall from  \cite{CS20} the geometric condition which the two compact sets $\Sigma\subset S$ have to satisfy:
\begin{condition}\label{c:geometric}
There exists a geometric constant $\bar{r}>0$ such that
for any $\delta>0$ the set
$$
\biggl\{x\in\mathcal{M}:\,\textrm{dist}(x,\Sigma)\geq \delta\textrm{ and }\textrm{dist}(x,S)\leq \bar{r}\delta\biggr\}
$$
is contained in a pairwise disjoint union of open sets, each contained in a single chart
$\Omega_k$.
\end{condition}

Recall that in Proposition \ref{p:stage} we impose a smallness-condition on the oscillation of our metric $g$. We now fix an atlas for $\mathcal M$ respecting this condition as follows. Fix an arbitrary atlas of finitely many charts $\Omega_k$. By compactness there exists $\gamma_0\geq 1$ such that
\[ \gamma_0^{-1}\mathrm{Id} \leq G\leq \gamma_0\mathrm{Id},\, \|G\|_{C^{1}{(\Omega_k)}} \leq \gamma_0 \]
on any $\Omega_k$,  where, as above, $G$ is the coordinate expression of $g$. If necessary, we then subdivide $\Omega_k$ to achieve $\mathrm{osc}_{\Omega_k} G<\sigma_0(\gamma_0)$. The charts in Definition \ref{d:adapted} and Condition \ref{c:geometric} are assumed to satisfy these assumptions.

With Proposition \ref{p:stage} we are now ready to state and prove our inductive proposition, analogous to the iteration Proposition 4.1 in \cite{CS20}. The main difference in the proof when compared to the one of \cite{CS20} is the choice of $\tau$ (when applying our Proposition \ref{p:stage}) and corresponding estimates on $h$.

\begin{proposition}\label{p:inductive}
 Let $0<\theta<\frac12, b>1$, $\sigma<\frac{\sigma_0}{4}$. There exists a constant $A_0=A_0(\theta, \sigma,b)\geq 1$, such that the following holds.

Let $\Sigma\subset S$ be compact subsets of $\mathcal{M}$ satisfying Condition \ref{c:geometric}.
Let $u\in C^{1, \theta}(\mathcal{ M})$ be an adapted short embedding with respect to $\Sigma$ such that $g-u^\sharp e=\rho^2(g+h)$ with $\rho\leq 1/4$ in $\mathcal{M}$, $\Sigma=\{\rho=0\}$, and in any chart $\Omega_k$
\begin{equation}\label{e:inductive-u-rho-assumption}
\begin{split}
|\nabla^2 u|\leq A\rho^{1-\frac{1}{\theta}}, &\quad |\nabla\rho|\leq A\rho^{1-\frac{1}{\theta}}, \\
|h|\leq \sigma,&\quad |\nabla h|\leq A\rho^{-\frac{1}{\theta}},
\end{split}
\end{equation}
for some $A\geq A_0$. Then there exists an adapted short embedding $\bar{u}\in C^{1, \theta'}(\mathcal{M})$ with respect to $S$ such that $g-\bar{u}^\sharp e=\bar{\rho}^2(g+\bar{h})$,  $\bar{\rho}\leq\rho$, $\|\bar{u}-u\|_0\leq A^{-1/2}$,  and $\bar{u}=u$, $d\bar u= du$ on $\Sigma$\footnote{The equality $du=d\bar u$ on $\Sigma$ is intended as an equality of sections of the bundle $T^{*}\mathcal M\to \Sigma$.}. Moreover, in any chart $\Omega_k$
\begin{equation}\label{e:inductive-conclusion}
\begin{split}
|\nabla^2 \bar{u}|\leq A'\bar{\rho}^{1-\frac{1}{\theta'}}, &\quad |\nabla\bar{\rho}|\leq A'\bar{\rho}^{1-\frac{1}{\theta'}}, \\
|\bar{h}|\leq \sigma',&\quad |\nabla \bar{h}|\leq A'\bar{\rho}^{-\frac{1}{\theta'}},
\end{split}
\end{equation}
with
$$A'=A^{b^2}, \quad \theta'=\frac{\theta}{b^2}, \quad  \sigma'=4\sigma.$$

\end{proposition}

\begin{proof}
As in \cite{CS20}, the proof is also  divided into three  steps.

%%%%%%%%%%
{\bf Step 1. Parameters, cut-off functions and error size sequence.}
This step is same as in \cite{CS20}. First, recall that on any chart it holds
$$
{\gamma_0}^{-1}\textrm{Id}\leq G\leq \gamma_0\textrm{Id},\quad \|G\|_{C^1(\Omega_k)}\leq \gamma_0, \quad \mathrm{osc}_{\Omega_k} G < \sigma_0(\gamma_0)\,,
$$
and let $\gamma:=4\gamma_0$. By $\rho<\frac14$ and the assumption that $u$ is an adapted embedding, it is easy to get
$$
{\gamma}^{-1}\textrm{Id}\leq \nabla u^T\nabla u\leq \gamma\textrm{Id}.
$$
Next, set
\begin{equation}\label{e:deltabar}
\delta_1:=\max_{x\in\mathcal{M}}\rho^2,
\end{equation}
and for $q\geq 1$
\begin{align*}
\lambda_q=A\delta_q^{-\frac{1}{2\theta}},\quad \lambda_{q+1}=\lambda_q^b.
\end{align*}
When $A$ is sufficiently large (depending on $\theta,\sigma$),   we have
\begin{equation}\label{e:ordering}
\delta_{q+1}\leq \tfrac{1}{4}\delta_q,\quad \lambda_{q+1}\geq 2\lambda_q.
\end{equation}
We also decompose $\mathcal{M}$ with respect to $\Sigma$ and $S$. Let
$$
r_q=A^{-1}\delta_{q+1}^{\tfrac{1}{2\theta}}=\lambda_{q+1}^{-1},
$$
and define for $q=0, 1,2,\dots$
\begin{align*}
S_q&=\{x:\,\textrm{dist}(x,S)<r_*r_q\},\\
\widetilde S_q&=\{x:\,\textrm{dist}(x,S)<\tilde r_*r_q\},\\
\Sigma_q&=\{x:\,\textrm{dist}(x,\Sigma)<r_{**}r_q\},
\end{align*}
where $r_*<\tilde r_{*}$ and $r_{**}$ are geometric constants to be chosen in the following order:
\begin{enumerate}
\item Choose $r_{**}>0$ so that
\begin{equation}\label{e:choiceofr**}
\rho(x)>\tfrac{3}{2}\delta_{q+2}^{1/2}\quad\textrm{ implies }\quad x\notin \Sigma_{q+1}.\footnote{Such a choice is possible, since \eqref{e:inductive-u-rho-assumption} implies that $\rho$ is H\"older continuous with exponent $\theta$.}
\end{equation}
%Indeed, recall from \eqref{e:adapt-final-2} and the discussion following it, that $\rho$ is $\theta$-H\"older continuous and hence
%$\rho(x)\leq A^{\theta}\textrm{dist}(x,S)^\theta$. In particular, for any $x\in S_{q+1},$ one has $\rho(x)\leq r_{**}^\theta\delta_{q+2}^{1/2}$. Thus such a choice of $r_{**}$ is possibl.e

\item Set $\tilde r_*=\bar{r}r_{**}$, where $\bar{r}>0$ is the geometric constant in Condition \ref{c:geometric}, which implies that for any $q\in\N$
\begin{equation}\label{e:singlecharts}
\begin{split}
	\widetilde S_q\setminus \Sigma _q&\textrm{ \emph{is contained in a pairwise disjoint union of open sets,}}\\
	&\textrm{\emph{each contained in a single chart }}\Omega_k.	
\end{split}
\end{equation}
\item Choose $r_*<\tilde r_*$ so that $\tfrac{1}{2}\tilde r_*<r_*<\tilde r_*$. Then we have by \eqref{e:ordering}
$$
\widetilde S_{q+1}\subset S_q\subset \widetilde S_q\quad\textrm{ for all }q.
$$
\end{enumerate}

Next, we fix cut-off functions $\phi, \tilde\phi, \psi, \tilde\psi\in C^\infty(0,\infty)$ with $\phi,\tilde\phi$ monotonic increasing, $\psi,\tilde\psi$ monotonic decreasing such that
$$
\phi(s),\tilde\phi(s)=\begin{cases} 1&s\geq 2\\ 0&s\leq \tfrac32\end{cases}\,,\quad
\psi(s),\tilde\psi(s)=\begin{cases} 1&s\leq r_*\\ 0&s\geq \tilde r_*\end{cases}\,,
$$
and in addition
$$
\tilde\phi(s)=1 \textrm{ on }\supp\phi,\quad \tilde\psi(s)=1 \textrm{ on }\supp\psi.
$$
As in \cite{CS20}, set
\begin{align*}
\chi_q(x)=\phi\left(\frac{\rho(x)}{\delta_{q+2}^{1/2}}\right)\psi\left(\frac{\textrm{dist}(x, S)}{r_{q+1}}\right),\quad
\tilde\chi_q(x)=\tilde\phi\left(\frac{\rho(x)}{\delta_{q+2}^{1/2}}\right)\tilde\psi\left(\frac{\textrm{dist}(x, S)}{r_{q+1}}\right).
\end{align*}
Using \eqref{e:inductive-u-rho-assumption} and the choice of $r_q$, $r_*$, $\tilde r_*$ and the cut-off functions we easily deduce
 \begin{align}
 |\nabla\chi_q|,\, |\nabla\tilde\chi_q|&\leq CA\delta_{q+2}^{-\frac{1}{2\theta}}=C\lambda_{q+2},\label{e:gradient-chi}\\
\textrm{dist}(\supp\chi_q, \partial\supp\tilde\chi_q)&\geq C^{-1}A^{-1}\delta_{q+2}^{\frac{1}{2\theta}}=C^{-1}\lambda^{-1}_{q+2}.\label{e:chi-q-support}
 \end{align}
for some constant $C$ depending on $r_*,\tilde r_*$, and moreover
\begin{equation}\label{e:chi-define}
\begin{split}
\{x\in S_{q+1}|\rho(x)>2\delta_{q+2}^{1/2}\}&\subset\{x\in\mathcal{M}:\,\chi_q(x)=1\},\\
\supp\chi_q&\subset \{x\in\mathcal{M}:\,\tilde{\chi}_q(x)=1\},\\
\supp\tilde{\chi}_q&	\subset\{x\in\widetilde S_{q+1}:\,\rho(x)>\tfrac{3}{2}\delta_{q+2}^{1/2}\}.
\end{split}
\end{equation}
From \eqref{e:choiceofr**}  and \eqref{e:singlecharts} we then deduce that $\supp\tilde\chi_q$ is contained in a pairwise disjoint union of open sets, each contained in a single chart $\Omega_k$.

%%%%%%%%%%%%%%%

Finally, we define  the sequence of error size $\{\rho_q\}.$
Set $\rho_0=\rho$ and define $\rho_{q}$ for $q=1,2,\dots$ inductively as
\begin{equation}\label{e:rho-q+1}
\rho_{q+1}^2=\rho_q^2(1-\chi_q^2)+\delta_{q+2}\chi_q^2.
\end{equation}
One can prove by induction (cf. Lemma 4.1 of \cite{CS20}) that the thus defined maps $\rho_q$ have  the following properties.
\begin{lemma}\label{l:rho}
Let $\{\rho_q\}$ be defined in \eqref{e:rho-q+1}. Then for any $q=0,1,\dots$
\begin{enumerate}
	\item[(i)] On $\supp\tilde\chi_q$ it holds
	\begin{equation*}%\label{e:rhoq-bound}
	\tfrac{3}{2}\delta_{q+2}^{1/2}\leq\rho_q\leq2\delta_{q+1}^{1/2}.
	\end{equation*}
	\item[(ii)] For every $x$ we have $\rho_{q+1}(x)\leq\rho_q(x)$.
	\item[(iii)] If $\rho_q(x)\leq\delta_{q+1}^{1/2}$, then $x\not\in\bigcup_{j=0}^{q-1}\supp\tilde\chi_j$ and consequently $\rho_q(x)=\rho(x)$.
	\item[(iv)] If $\rho_q(x)\geq \delta_{q+1}^{1/2}$, then either $\chi_q(x)=1$ or $x\notin S_{q+1}$.
	\end{enumerate}
\end{lemma}

Now we are ready to inductively construct  a sequence of adapted short embeddings.
%%%%%%%%%%%%%%%

{\bf Step 2. Inductive construction}
%%%%%%%%%%%%%%%%%%
This step is similar to that of Proposition 4.1 in \cite{CS20}, but we need to pay more attention to the choice of $\tau$ and the estimate of $h$. We will construct a sequence of smooth adapted short embeddings $(u_q, \rho_q, h_q)$ such that the following hold:
\begin{itemize}
\item[$(1)_q$] For all $\mathcal{M},$ we have
$$g-u_q^\sharp e=\rho_q^2(g+h_q).$$
\item[$(2)_q$] If $x\notin\bigcup_{j=0}^{q-1}\supp\tilde\chi_j$, then $(u_q, \rho_q, h_q)=(u_0, \rho_0, h_0)$ and $ du_q=du_0$ along $\Sigma.$
\item[$(3)_q$] The following estimates hold in $\mathcal{M}$:
\begin{align}
|\nabla^2u_q|\leq A^{b^2}\rho_q^{1-\frac{b^2}{\theta}},\quad & |\nabla\rho_q|\leq A^{b^2}\rho_q^{1-\frac{b^2}{\theta}} ,\label{e:inductive-v-rho-j}\\
|h_q|\leq 4\sigma, %A^{-\frac{\theta\alpha}{2b^2}}\rho_q^{\frac{\alpha}{2b^2}},
\quad &|\nabla h_q|\leq A^{b^2}\rho_q^{-\frac{b^2}{\theta}},
\label{e:inductive-h-j}
\end{align}
\item[$(4)_q$] On $\{x: \rho_0(x)>\delta_{q+1}^{1/2}\}\cap S_q,$  we have the sharper estimates
\begin{align}
|\nabla^2u_q|\leq A^{b}\rho_q^{1-\frac{b}{\theta}}, \quad &|\nabla\rho_q|\leq A^b\rho_q^{1-\frac{b}{\theta}} ,\label{e:inductive-rho-q}\\
|h_q|\leq \sigma,
\quad &|\nabla h_q|\leq A^{b}\rho_q^{-\frac{b}{\theta}}.
\label{e:inductive-h-q}
\end{align}
\item[$(5)_q$] We have the global estimate for $q\geq1$
\begin{align}
&\|u_q-u_{q-1}\|_0\leq \overline{C}\delta_q^{1/2}\lambda_{q}^{-1},\label{e:inductive-v-difference-0}\\
&\|u_q-u_{q-1}\|_1\leq\overline{C}\delta_q^{1/2},\label{e:inductive-v-difference-1}
\end{align}
where $\overline{C}$ is the constant in the conclusions of Proposition \ref{p:stage} in \eqref{e:stage-2}-\eqref{e:stage-3}.
\end{itemize}

{\it Initial step $q=0$.}  Set $(u_0, \rho_0, h_0)=(u, \rho, h).$ Since $b>1,$ it is easy to check $(1)_0-(2)_0$ and $(4)_0$ from \eqref{e:inductive-u-rho-assumption}.

{\it Inductive step $q\mapsto q+1$. } Suppose $(u_q, \rho_q, h_q)$  is an adapted short embedding on $\mathcal{M}$ satisfying $(1)_q-(5)_q.$ We then construct $(u_{q+1}, \rho_{q+1}, h_{q+1}).$ In fact, $\rho_{q+1}$ has already been defined in \eqref{e:rho-q+1}.
We shall estimate $(u_q, \rho_q, h_q)$ on $\supp\tilde\chi_q.$
As derived in \cite{CS20}, on $\supp\tilde\chi_q,$ we have
\begin{equation}\label{e:support-rho_q}
\begin{split}
\tfrac{3}{2}\delta_{q+2}^{1/2}&\leq \rho_q\leq2\delta_{q+1}^{1/2},\\
 |\nabla \rho_q|&\leq \delta_{q+1}^{1/2}\lambda_{q+2} ,\, |\nabla^2u_q|\leq \delta_{q+1}^{1/2}\lambda_{q+2},\,\left|\frac{\nabla\rho_q}{\rho_q}\right|\leq \lambda_{q+2}\\
|h_q|&\leq\sigma, \,\, |\nabla h_q|\leq \lambda_{q+2}.
\end{split}
\end{equation}
We then want to apply Proposition \ref{p:stage} to construct $(u_{q+1}, h_{q+1}).$ To this end define
\begin{align*}
\tilde{\rho}_q=\chi_q\sqrt{\rho_q^2-\delta_{q+2}},\quad \tilde h_q=\frac{\tilde\chi_q\rho_q^2}{\rho_q^2-\delta_{q+2}}h_q\,.
\end{align*}
From \eqref{e:support-rho_q}, one has on $\supp\tilde\chi_q$
$$\tfrac54\delta_{q+2}\leq\rho_q^2-\delta_{q+2}\leq4\delta_{q+1},$$
hence $\tilde\rho_q$ and $ \tilde h_q$ are well defined. Note that with these definitions, we then have
$$\tilde\rho_q^2(g+\tilde h_q)=\chi_q^2((\rho_q^2-\delta_{q+2})g +\rho_q^{2}h_q)=\chi_q^2(g-u_{q}^\sharp e-\delta_{q+2}g)$$
using that $\tilde \chi_q = 1$ on the support of $\chi_q$ and the inductive assumption $(1)_q$. Thus, by adding the tensor $\tilde \rho_q^{2}(g+\tilde h_q)$ we will be able to get a map $u_{q+1}$ which is, upto an error of size $\delta_{q+2}$, isometric on the support of $\chi_q$.

 We therefore want to estimate $\tilde \rho_q$ and $\tilde h_q$ and choose $\delta, \lambda$ in Proposition \ref{p:stage} accordingly.  We thus set $\Omega = \supp\tilde\chi_q,$ and observe
 \begin{align*}
 |\nabla\sqrt{\rho_q^2-\delta_{q+2}}|&\leq C|\nabla\rho_q|,\\
\frac{\rho_q^2}{\rho_q^2-\delta_{q+2}}&=1+\frac{\delta_{q+2}}{\rho_q^2-\delta_{q+2}}\leq2,\\
\left|\nabla\frac{\rho_q^2}{\rho_q^2-\delta_{q+2}}\right|&=\left|\nabla\frac{\delta_{q+2}}{\rho_q^2-\delta_{q+2}}\right|
\leq C\left|\frac{\nabla\rho_q}{\rho_q}\right|,
 \end{align*}
where $C$ are geometric constants. Therefore, using \eqref{e:gradient-chi} and \eqref{e:support-rho_q} we can infer
\begin{equation}\label{e:rho-tilde}
\begin{split}
0\leq\tilde\rho_q&\leq\rho_q\leq2\delta_{q+1}^{1/2},\, |\tilde h_q|\leq2\sigma\leq\frac{\sigma_0}{2},\\
|\nabla\tilde\rho_q|&\leq C(|\nabla\chi_q|\rho_q+|\nabla\rho_q| \leq C\delta_{q+1}^{1/2}\lambda_{q+2},\\
|\nabla\tilde h_q|&\leq C(|\nabla\tilde\chi_q||h_q|+\left|\frac{\nabla\rho_q}{\rho_q}\right||h_q|+|\nabla h_q|)\leq C\lambda_{q+2}.
\end{split}
\end{equation}
Therefore $(u_q, \tilde\rho_q, \tilde h_q)$ satisfies all the assumptions in Proposition \ref{p:stage} on $\supp\tilde\chi_q$ with $\delta, \lambda$ given by $4\delta_{q+1}, C\lambda_{q+2}$ respectively. Setting \[\tau=1+\frac{1-\theta}{b}(b-1)>1\,,\]
we only need to make sure that $C\lambda_{q+2} \geq \lambda_0(\gamma,\sigma_0, \tau)$ in \eqref{e:stageparameters}. This however follows by choosing $A\geq A_0(\theta,\sigma,b)$ large enough.

Thus, recalling \eqref{e:singlecharts} that  $\supp\tilde\chi_q$ is contained in a pairwise disjoint union of open sets, each contained in a single chart,  we may apply Proposition \ref{p:stage} in each open set separately in local coordinates to add the term $\tilde\rho_q^2(g+\tilde h_q)$.
Overall we obtain $u_{q+1}$ and $\mathcal{E}$ such that
$$g-u_{q+1}^\sharp e=(g-u_q^\sharp e)(1-\chi_q^2)+\delta_{q+2}g\chi_q^2+\mathcal{E}.$$
with $u_{q+1}$ satisfying
\begin{equation}\label{e:v-q+1-c2}
|\nabla^2u_{q+1}|\leq C\delta_{q+1}^{1/2}\lambda_{q+2}^{\tau}=C\delta_{q+1}^{1/2}\lambda_{q+1}^{b+(1-\theta)(b-1)},
\end{equation}
and $\mathcal{E}$ satisfying
\begin{align}
&|\mathcal{E}|\leq C\delta_{q+1}\lambda_{q+2}^{2-2\tau}=C\delta_{q+2}\lambda_{q+1}^{-2(1-2\theta)(b-1)}, \label{e:error-q+1--1}\\
&|\nabla\mathcal{E}|\leq C\delta_{q+1}\lambda_{q+2}=C\delta_{q+2}\lambda_{q+1}^{b+2\theta(b-1)}, \label{e:error-q+1--2}
\end{align}
which are implied by $\delta_{q+1}=\lambda_{q+1}^{2\theta(b-1)}\delta_{q+2}$ and \eqref{e:stage-4}-\eqref{e:stage-5}. From  \eqref{e:stage-support}, one gets
$$\supp(u_{q+1}-u_q), \,\supp\mathcal{E}\subset\supp\chi_q+B_{\kappa_q}(0),$$
with
$$
\kappa_q=(C\lambda_{q+2})^{1-2\tau}\leq \lambda_{q+1}^{-2(1-\theta)(b-1)}\lambda_{q+2}^{-1}\leq C^{-1}\lambda_{q+2}^{-1},
$$
where $C$ is the constant in \eqref{e:chi-q-support} and the last inequality holds provided $A$ is sufficiently large. Consequently $u_{q+1}=u_q,, du_{q+1}=du_q$ and $\mathcal{E}=0$ outside  $\supp\tilde\chi_q$.

Moreover, \eqref{e:inductive-v-difference-0} and \eqref{e:inductive-v-difference-1} for the case $q+1$ follow immediately from \eqref{e:stage-2}-\eqref{e:stage-3}, hence $(5)_{q+1}$ is verified. We also define
$$h_{q+1}=(1-\chi_q^2)\frac{\rho_q^2}{\rho_{q+1}^2}h_q+\frac{\mathcal{E}}{\rho_{q+1}^2}$$
so that $$g-u_{q+1}^\sharp e=\rho_{q+1}^2(g+h_{q+1}),$$
verifying $(1)_{q+1}.$  Note that on $\supp\tilde\chi_q$ using \eqref{e:support-rho_q} one has
\begin{equation}\label{e:rhoq+1-bound}
\begin{split}
\rho_{q+1}^2&\leq4\delta_{q+1}(1-\chi_q^2)+\delta_{q+2}\chi_q^2\leq 4\delta_{q+1},\\
\rho_{q+1}^2&\geq\tfrac94\delta_{q+2}(1-\chi_q^2)+\delta_{q+2}\chi_q^2\geq\delta_{q+2}.
\end{split}
\end{equation}
Thus $h_{q+1}$ is well defined. Besides we can also derive that $(\rho_{q+1}, h_{q+1})$ agrees with $(\rho_q, h_q)$ outside $\supp\tilde\chi_q.$ It remains to verify $(2)_{q+1}-(4)_{q+1}$ on $\supp\tilde\chi_q.$

{\it Verification of $(2)_{q+1}$} If $x\not\in\bigcup_{j=0}^{q}\supp\tilde\chi_j$, then  $\tilde\chi_q(x)=0$ and therefore
$$(u_{q+1}, \rho_{q+1}, h_{q+1})=(u_q, \rho_q, h_q)=(u_0, \rho_0, h_0).$$

{\it Verification of $(3)_{q+1}$}  On $\supp\tilde\chi_q,$ we first calculate
\begin{equation}\label{e:gradient-rho-q+1}
\begin{split}
|\nabla\rho_{q+1}|&=\frac{|\nabla\rho_{q+1}^2|}{2\rho_{q+1}}\leq\frac{C}{\rho_{q+1}}(|\rho_q\nabla\rho_q|
+|\nabla\chi_q|(\rho_q^2+\delta_{q+2}))\\
&\leq C\frac{\delta_{q+1}\lambda_{q+2}}{\delta_{q+2}^{1/2}}=CA^{b+(b-1)\theta}\delta_{q+1}^{1-\frac{b}{2}(1+\frac{1}{\theta})}\\
&\leq A^{b^2}(2\delta_{q+1}^{1/2})^{1-\frac{b^2}{\theta}}\leq A^{b^2}\rho_{q+1}^{1-\frac{b^2}{\theta}},
\end{split}
\end{equation}
where we have used \eqref{e:gradient-chi}, \eqref{e:support-rho_q} and \eqref{e:rhoq+1-bound}. For the inequality in the last line we have used that $1-\frac{b}{2}(1+\frac{1}{\theta})\geq \frac{1}{2}(1-\frac{b^2}{\theta})$, $2(b-1)\theta+b\leq b^2$ ( from $b>1$ and $2\theta<1$) and $A$ sufficiently large to absorb geometric constants.

Similarly, using \eqref{e:support-rho_q}, \eqref{e:v-q+1-c2}-\eqref{e:error-q+1--1} and \eqref{e:rhoq+1-bound} we obtain
\begin{equation}\label{e:hq+1-bound-v-c2}
\begin{split}
|h_{q+1}|&\leq|h_q|+\frac{|\mathcal{E}|}{\rho_{q+1}^2}\leq 2\sigma+C\lambda_{q+1}^{-2(1-2\theta)(b-1)}\leq 3\sigma,\\
|\nabla^2u_{q+1}|&\leq C\delta_{q+1}^{1/2}\lambda_{q+1}^{b+(1-\theta)(b-1)}\leq C\delta_{q+1}^{1/2}\lambda_{q+1}^{b^2-\theta(b-1)}\leq CA^{b^2-\theta(b-1)}\delta_{q+1}^{\frac{1}{2}(1-\frac{b^2}{\theta})}\\
&\leq A^{b^2}\rho_{q+1}^{1-\frac{b^2}{\theta}},
\end{split}
\end{equation}
where we have used $(1-\theta)(b-1)+\theta(b-1)\leq b^2-b$ (by $b>1$) and again assumed $A$ sufficiently large to absorb the constants $C$.
For $|\nabla h_{q+1}|,$ we calculate as follows.
\begin{equation}\label{e:gradient-h-q+1}
\begin{split}
|\nabla h_{q+1}|&\leq|\nabla h_q|+\frac{1}{\rho_{q+1}^2}(|\nabla\mathcal{E}|+\delta_{q+2}|\nabla(h_q\chi_q^2)|)+\frac{2|\nabla\rho_{q+1}|}{\rho_{q+1}^3}(\delta_{q+2}|h_q|+|\mathcal{E}|)\\
&\leq C\lambda_{q+2}+C\lambda_{q+1}^{b+2\theta(b-1)}+C\frac{\delta_{q+1}\lambda_{q+2}}{\delta_{q+2}}
(\sigma+\lambda_{q+1}^{-2(1-\theta)(b-1)})\\
&\leq C\lambda_{q+2}+C\lambda_{q+1}^{b+2\theta(b-1)}+C\frac{\delta_{q+1}}{\delta_{q+2}}\lambda_{q+2}\\
&\leq C\lambda_{q+1}^{b+2\theta(b-1)},
\end{split}
\end{equation}
where we have used \eqref{e:gradient-chi}, \eqref{e:support-rho_q}, \eqref{e:error-q+1--1}, \eqref{e:error-q+1--2} and \eqref{e:gradient-rho-q+1}.
Using again the inequality $b+2\theta(b-1)<b^2-(1-2\theta)(b-1)$,  we further estimate
\begin{equation}\label{e:gradient-h-q+1-0}
\begin{split}
|\nabla h_{q+1}|&\leq C\lambda_{q+1}^{b^2-(1-2\theta)(b-1)} \leq  CA^{b^2-(1-2\theta)(b-1)}\delta_{q+1}^{-\frac{b^2}{2\theta}}\\
&\leq A^{b^2}\rho_{q+1}^{-\frac{b^2}{\theta}},
\end{split}
\end{equation}
where we have again used that $A$ is sufficiently large. Thus we have shown \eqref{e:inductive-v-rho-j} for $q+1,$ i.e. $(3)_{q+1}$ is verified.

{\it Verification of $(4)_{q+1}$} Observe that by \eqref{e:chi-define}
\begin{align*}
\{x\in S_{q+1}: \rho_0(x)>\delta_{q+2}^{1/2}\}
=\{\chi_q(x)=1\}\cup\{x\in S_{q+1}: \delta_{q+2}^{1/2}\leq\rho_0(x)\leq2\delta_{q+2}^{1/2}\}.
\end{align*}
If $x\in\{\chi_q=1\},$ then
$$\rho_{q+1}=\delta_{q+2}^{1/2},\quad  h_{q+1}=\frac{\mathcal{E}}{\delta_{q+2}}.$$
Using \eqref{e:v-q+1-c2},
\begin{equation}\label{e:vq+1}
|\nabla^2 u_{q+1}|\leq C\delta_{q+1}^{1/2}\lambda_{q+1}^{b+(1-\theta)(b-1)}= C\delta_{q+1}^{1/2}\lambda_{q+1}^{2b-\theta(b-1)-1}\leq CA^{2-\frac{1}{b}-b}A^b\delta_{q+2}^{\frac{1}{2}(1-\frac{b}{\theta})}.
\end{equation}
where we have used $2-\frac{1}{b}<b$.  By taking $A$ sufficiently large we absorb the geometric constant $C$ and deduce \eqref{e:inductive-rho-q}.

In order to verify \eqref{e:inductive-h-q} we calculate using \eqref{e:error-q+1--1}-\eqref{e:error-q+1--2}:
\begin{equation*}
\begin{split}
|h_{q+1}|&\leq C\lambda_{q+2}^{-2(1-2\theta)(b-1)}\leq\sigma,\\
|\nabla h_{q+1}|&\leq C\lambda_{q+1}^{b+2\theta(b-1)}\leq \lambda_{q+2}^{b}=A^{b}\delta_{q+2}^{-\frac{b}{2\theta}}
\end{split}
\end{equation*}
using $b+2\theta(b-1)<b^2$ By choosing $A$ sufficiently large, we can then absorb again the geometric constants and conclude  \eqref{e:inductive-h-q}. Hence $(4)_{q+1}$ is obtained for this case.

\smallskip

On the other hand, if $x\in\{x\in\Sigma_{q+1}: \delta_{q+2}^{1/2}\leq\rho_0(x)\leq2\delta_{q+2}^{1/2}\},$
then $(u_q, \rho_q, h_q)=(u_0, \rho_0, h_0)$ by $(2)_q$ and $\rho_0\leq2\delta_{q+2}^{1/2}$.
Thus
\begin{equation}\label{e:rho-q+1-bound}
\begin{split}
\rho_{q+1}^2&\geq\delta_{q+2}(1-\chi_q^2)+\delta_{q+2}\chi_q^2\geq\delta_{q+2},\\
\rho_{q+1}^2&\leq4\delta_{q+2}(1-\chi_q^2)+\delta_{q+2}\chi_q^2\leq4\delta_{q+2}.
\end{split}
\end{equation}
Therefore, choosing again $A$ sufficiently large to absorb geometric constants,
\begin{equation}
\label{e:h-q+1-bound}
\begin{split}
|h_{q+1}|&\leq |h_0|+\left|\frac{\mathcal{E}}{\rho_{q+1}^2}\right|\\
&\leq \sigma+C\lambda_{q+1}^{-2(1-2\theta)(b-1)}\\
&\leq 2\sigma.
\end{split}
\end{equation}
Moreover, calculating as in \eqref{e:gradient-rho-q+1} but this time using \eqref{e:rho-q+1-bound}
\begin{equation*}
\begin{split}
|\nabla\rho_{q+1}|&=\frac{|\nabla\rho_q^2|}{2\rho_{q+1}}\leq\frac{C}{\rho_{q+1}}(|\rho_q\nabla\rho_q|
+|\nabla\chi_q|(\rho_q^2+\delta_{q+2}))\\
&\leq C\delta_{q+2}^{1/2}\lambda_{q+2}=CA\delta_{q+2}^{\frac{1}{2}(1-\frac{1}{\theta})}\\
&\leq A^{b}\delta_{q+2}^{\frac{1}{2}(1-\frac{b}{\theta})}\leq A^{b}\rho_{q+1}^{1-\frac{b}{\theta}}.
\end{split}
\end{equation*}
Similarly, proceeding as in \eqref{e:gradient-h-q+1}-\eqref{e:gradient-h-q+1-0} we have
\begin{equation*}
|\nabla h_{q+1}|\leq C\lambda_{q+2}^{1+2\theta(1-\frac{1}{b})}=CA^{1+2\theta(1-\frac{1}{b})}\delta_{q+2}^{-\frac{1}{2\theta}-(1-\frac{1}{b})}\leq A^{b}\rho_{q+1}^{-\frac{b}{\theta}}.
\end{equation*}
Finally, the estimate for $\nabla^2u_{q+1}$ has already been obtained in \eqref{e:vq+1}.
Therefore $(4)_{q+1}$ is verified also in this case.

Overall we have shown that $(u_{q+1}, \rho_{q+1}, h_{q+1})$ satisfies $(1)_{q+1}-(5)_{q+1}.$

%%%%%%%%%%%%%%
\noindent{\bf Step 3. Conclusion}
%%%%%%%%%%%%%%

We are now in a position to take the limit as $q\rightarrow\infty.$ Recalling \eqref{e:ordering} we see that $\delta_q^{1/2}\leq 2^{-q-1}$ and $\delta_q^{1/2}\lambda_q^{-1}\leq A^{-1}2^{-q-1}$. In particular from $(5)_q$ we see that $\{u_q\}$ is a Cauchy sequence in $C^1(\mathcal{M})$.

From the formula \eqref{e:rho-q+1} and Lemma \ref{l:rho} we deduce $0\leq \rho_{q}-\rho_{q+1}\leq 2\delta_{q+1}^{1/2}$, so that $\{\rho_q\}$ is a Cauchy sequence in $C^0(\mathcal{M})$. From $(1)_q-(3)_q$ we can also deduce that $\{h_q\}$ is a Cauchy sequence in $C^0(\mathcal{M})$; indeed, this follows from the formula  $(1)_q$, the fact that $u_q^\sharp e$ and $\rho_q^2$ are Cauchy sequences, and \eqref{e:inductive-h-j}.

Furthermore, since $\supp\tilde\chi_q\subset S_q$ and $\bigcap_q S_q=S$, using $(2)_q$ we see that for any $x\in \mathcal{M}\setminus S$ there exists $q_0=q_0(x)$ such that
$$
(u_{q}, \rho_{q}, h_{q})=(u_{q_0}, \rho_{q_0}, h_{q_0})
$$
for all $q\geq q_0(x)$. Similarly, since $\supp\tilde\chi_q\subset\{\rho>\delta_{q+1}^{1/2}\}$, $(u_q, \rho_q, h_q)$ agrees with $(u,\rho,h)$ on $\Sigma$. Thus there exist
\begin{align*}
\bar{u}&\in C^1(\mathcal{M})\cap C^2(\mathcal{M}\setminus S),\\
\bar{\rho}&\in C^0(\mathcal{M})\cap C^1(\mathcal{M}\setminus S),\\
\bar{h}&\in C^0(\mathcal{M}, \R^{2\times2})\cap C^1(\mathcal{M}\setminus S, \R^{2\times 2}),
\end{align*}
such that
$$
u_q\rightarrow \bar{u}, \quad u_q^\sharp e\rightarrow \bar{u}^\sharp e, \quad \rho_q\rightarrow\bar{\rho},\quad h_q\rightarrow \bar{h}\textrm{ uniformly on }\mathcal{M}.
$$
The limit $(\bar{u}, \bar{\rho}, \bar{h})$ satisfies
$$
g-\bar{u}^\sharp e=\bar{\rho}^2(g+\bar{h})\textrm{ on }\mathcal{M}
$$
using $(1)_q$. By $(2)_q,$ $\bar u=u$ and $ d\bar u=du$ on $\Sigma.$ Moreover, we have
$$
\|\bar{u}-u\|_0\leq \sum_{q=1}^\infty\|u_q-u_{q-1}\|_0\leq \overline{C}A^{-1}\sum_{q=1}^\infty2^{-q-1}= \frac{1}{2}\overline{C}A^{-1}\leq A^{-1/2}
$$
using $(5)_q$ and ensuring $A$ is large enough to absorb the constant $\overline{C}$,
and, using $(3)_q$,
\begin{align*}
|\nabla^2\bar{u}|\leq A^{b^2}\bar{\rho}^{1-\frac{b^2}{\theta}},\quad & |\nabla\bar{\rho}|\leq A^{b^2}\bar{\rho}^{1-\frac{b^2}{\theta}},\\
|\bar{h}|\leq 4\sigma,\quad &|\nabla \bar{h}|\leq A^{{b^2}}\bar{\rho}^{-\frac{b^2}{\theta}}.
\end{align*}
Finally, from Lemma \ref{l:rho} and \eqref{e:chi-define} we see that $\rho_q\leq 2\delta_{q+1}^{1/2}$ on $S$. Combined with the observation above that for any $x\notin S\supset \Sigma$ we have $\bar{\rho}(x)=\rho_q(x)>0$ for some $q$, we deduce $\{\bar{\rho}=0\}=\Sigma$. This proves that $(\bar{u}, \bar{\rho}, \bar{h})$ is an adapted short embedding with respect to $S\supset\Sigma $ with exponent $\theta'=\frac{\theta}{b^2}$, and satisfying \eqref{e:inductive-conclusion} as required. The proof of Proposition \ref{p:inductive} is completed.
\end{proof}
%%%%%%%%%%%

\bigskip

%%%%%%
\section{Proof of Theorem \ref{t:rigidity-flexibility}~(2): Flexibility part}\label{s:flexibility}
%%%%%%%%%%

The goal of this section is to show the flexibility part of Theorem \ref{t:rigidity-flexibility}. The proof is divided into three steps.

{\bf Step 1. Short extension.} In the first step we want to construct an embedding which is isometric on $\Sigma$ and strictly short on $\Sigma_\epsilon^{+}\setminus \Sigma$ for a one-sided neighborhood $\Sigma_\epsilon^{+}\subset M $.
 The construction is analogous to the one in  \cite{CS19}  (see also \cite{HunWas16})  except that we want to define  $u$ not only locally around a point $p\in \Sigma$.

Recall that the one-sided neighborhood is defined as $\Sigma_\epsilon^{+} = F(\Sigma\times [0,\epsilon[)$ for $F:\Sigma\times ]-\epsilon_0,\epsilon_0[ \to \mathcal M$ given by $F(p,t) = \exp_p(t\nu(p))$. We then define our short extension $u:\Sigma_\epsilon^{+}\to \R^{m}$ by
%%To do this we fix $\epsilon_0>0$ such that the map $F:\Sigma\times ]-\epsilon_0,\epsilon_0[\, \to \mathcal M$ given by
%%\begin{equation}\label{d:F}
%%F(p,t) = \exp_p(t\nu(p))
%%\end{equation}
%is a diffeomorphism onto its image, a tubular neighborhood which we denote by $\Sigma_{\epsilon_0}$. Here, $\nu$ is the unit normal vectorfield to $\Sigma$ and $\exp$ is the exponential map. For $\epsilon<\epsilon_0$ we then set $\Sigma_\epsilon^{+}:= F(\Sigma\times [0,\epsilon[)$ and define  our short extension $u:\Sigma_\epsilon^{+} \to \R^{m}$ by
\[ u(F(p,t)) = f(p)+ t\mu(p)-t^{2}\mu(p)\,.\]
We claim that $u$ is isometric on $\Sigma$ and strictly short on $\Sigma_\epsilon^{+}\setminus \Sigma$ if $\epsilon$ is small enough.
Indeed, fix a finite atlas $\{(V_i,\psi_i)\}_{i=1}^{N}$ for the manifold $\Sigma$ and extend it to $\Sigma_\epsilon$ using $F$. More precisely, set $U_i=F(V_i,]-\epsilon_0,\epsilon_0[)$ and define $\varphi_i :U_i\to \R^{n}$ by \[ \varphi_i(F(p,t)) = (\psi_i(p),t)\,.\]
Clearly in these coordinates it holds $\Sigma = \{ t=0\}$, and one can check that the metric in each $U_i$ is of the form
$$
g=\sum_{i, j=1}^{n-1}g_{ij}dx^idx^j+(dt)^2.
$$
Moreover, the scalar second fundamental form of the inclusion $\iota:\Sigma \hookrightarrow \mathcal M $ is given by
\[ L_{ij}(x) = -\frac{1}{2}\partial_t g_{ij}(x,0) \,.\]
By expanding $g_{ij}$ around $t=0$ we then get
\[ g_{ij}(x,t) = g_{ij}(x,0)-2tL_{ij}(x) + O(t^{2})\,.\]
On the other hand, we compute
\[ \langle \partial_i u,  \partial_j u  \rangle =  \langle \partial_i f , \partial_j f \rangle  + t\left ( \langle \partial_if ,\partial_j \mu \rangle  +  \langle \partial_jf , \partial_i \mu \rangle  \right )+ O(t^{2}) \,,\]
and
\[  \langle \partial_i   u, \partial_t u  \rangle = 0 \,, \langle  \partial_t u ,\partial_t u  \rangle = (1-2t)^{2}\,\] thanks to the properties of $\mu$. Since $f$ is an isometry and
\[ \langle  \partial_i f,\partial_j \mu \rangle  =  \langle  \partial_j f,\partial_i \mu \rangle = -\langle \bar L_{ij},\mu \rangle \]
we therefore get

\begin{equation}\label{e:metricadapt1}
g-\nabla u^T\nabla u=2t\left(
  \begin{array}{cc}
\langle \bar L_{ij},\mu \rangle - L_{ij}~&~0\\
0~&~2
  \end{array}
\right)+O(t^2)\,.
\end{equation}
Clearly, $u^{\sharp}e = g$ on $\Sigma$. Moreover, if $\epsilon>0$ is small enough, assumption \eqref{e:HW} implies that there exists $C\geq 1$ such that
\[ \left (g- \nabla u^{T}\nabla u\right )|_{(x,t)}\geq  C^{-1}t \mathrm{Id} \,\]
on $\Sigma_\epsilon^{+}$, showing the strict shortness of $u$ on $\Sigma_\epsilon^{+}\setminus \Sigma$.

Lastly, we observe that for $p\in \Sigma$ it holds
 \[  du_{p}(\nu) = \partial_t u ( F(p,0)) = \mu(p)\]
and consequently for any $X\in T_p\Sigma\setminus \{0\}$
\begin{equation}\label{e:flexibility} \langle du(\nu), \bar L(X,X)\rangle = \langle \mu,\bar L(X,X)\rangle >  L(X,X)\,.\end{equation}

{\bf Step 2. Adapted short extension.} Given the short extension $u$ from Step 1, we want to construct an adapted short embedding $v$ with $u=v$ and $du=dv$ on $\Sigma$. The step is similar to corresponding construction in \cite{CS19}, the main differences being the choice of frequency parameter below to make our extension of class $C^{1, \theta_0}$ for any $\theta_0<\frac12$ and the global nature of the present construction.  We use one stage of adding primitive metric errors to construct an adapted short embedding.  Choose $\gamma, M >1$ such that the short extension $u:\Sigma_\epsilon^{+}\to\R^{m}$ constructed in Step 1  satisfies $u\in C^2(\Sigma_\epsilon^{+})$ with
\begin{align*}
{\gamma}^{-1}\textrm{Id}\leq&\nabla u^T\nabla u\leq\gamma \textrm{Id},\\
&\|u\|_{C^2(U_i)}\leq M
\end{align*}
in every chart $U_i$. We then define
$$
\rho^2(x,t)=\frac{1}{n}\textrm{tr}(g-\nabla u^T\nabla u).
$$
Observe that this is a well-defined function on $\Sigma_\epsilon^{+}$ since the trace is invariant under coordinate transformations. By \eqref{e:metricadapt1}, we can seek a constant $C\geq 1$ so that
for all $(x,t)\in\Sigma_\epsilon^{+}$
\begin{equation}\label{e:adapted-rho}
{C}^{-1}t^{1/2}\leq \rho(x,t)\leq C t^{1/2},\quad |\nabla\rho(x,t)|\leq Ct^{-1/2},\quad |\nabla^2\rho(x,t)|\leq Ct ^{-3/2}.
\end{equation}
Furthermore, there exists $\alpha>0$ such that
\begin{align*}
g-\nabla u^T\nabla u\geq {C}^{-1}\rho^2\textrm{Id}\geq 2\alpha\rho^2g
\end{align*}
in every chart. We assume without loss of generality that $\alpha \rho^{2}\leq \frac{1}{16}$ on $\Sigma_\epsilon^{+}$ and $\alpha<1$. In particular, using Lemma 1 from \cite{Nash54} (see also Lemma 1.9 in \cite{SzLecturenotes}), we obtain the decomposition
\begin{align*}
\frac{g-\nabla u^T\nabla u}{\rho^2}-\alpha g=\sum_{k=1}^{\tilde N} \bar{b}_{k,i}^2\varpi_{k,i}\otimes\varpi_{k,i}
\end{align*}
in $U_i$, where $\varpi_{k,i}\in \mathbb{S}^{n-1}$, $\bar{b}_{k,i}\in C^{\infty}(U_i)$ and $\tilde N\in \N$, with estimates of the form
\begin{equation}\label{e:bkbar}
\|\bar{b}_{k,i}\|_{C^j(U_i)}\leq C
\end{equation}
for $j=0,1,2$.
Setting $b_k=\bar{b}_k\rho$ we derive
$$
g-\nabla u^T\nabla u-\alpha\rho^2g= \sum_{k=1}^{\tilde N}b_{k,i}^2\varpi_{k,i}\otimes\varpi_{k,i}
$$
in $U_i$, with estimates, for $j=0,1,2$ and $k=1,\dots,\tilde N$,
\begin{equation}\label{e:bk}
|\nabla^jb_{k,i}(x,t)|\leq Ct^{1/2-j}\quad \textrm{ for }(x,t)\in U_i.
\end{equation}
Now we define a Whitney-decomposition of $\Sigma_\epsilon^{+}\setminus\Sigma$  as follows: Set $d_q=2^{-q}\epsilon$ for $q=1,2,\dots$ and define
$$
\Sigma_q^{i}=F(V_i,]d_{q+1},d_{q-1}[) = U_i\cap \left (\Sigma_{d_{q-1}}^{+}\setminus \overline{ \Sigma_{d_{q+1}}^{+}} \right )\,.
$$
We then let $\{\chi_q^{i}\}_{q,i}$ be a partition of unity on $\Sigma^{+}_\epsilon\setminus \Sigma$ subordinate to the decomposition $\Sigma_\epsilon^{+}\setminus \Sigma = \bigcup_{q=1}^\infty\bigcup_{i=1}^{N}\Sigma_q^{i}$ with the following standard properties:
\begin{itemize}
\item[(a)] $\textrm{supp}\chi_q^{i}\subset \Sigma_q^{i}$, in particular $\textrm{supp}\chi_q^{i}\cap \textrm{supp}\chi_{q+2}^{i}=\emptyset$;
\item[(b)] $\sum_{i=1}^{N}\sum_{q=0}^\infty(\chi^i_q)^{2}=1$ in $\Sigma_\epsilon^{+}\setminus \Sigma$;
\item[(c)] For any $q,i$ and $j=0,1,2$ we have $\|\chi_q^{i}\|_{C^j(\Sigma_q^{i})}\leq Cd_q^{-j}$.
\end{itemize}
Consequently we can write
\begin{align}\label{e:gdecompevenodd}
g-u^{\sharp} e-\alpha\rho^2g= & \sum_{i=1}^{N}\sum_{k=1}^{\tilde N} \sum_{q\textrm{ odd}}(\chi_q^{i}b_{k,i})^2\varpi_{k,i}\otimes\varpi_{k,i}\\
\,\,\,\,&+\sum_{i=1}^{N}\sum_{k=1}^{\tilde N} \sum_{q\textrm{ even}}(\chi_q^{i}b_{k,i})^2\varpi_{k,i}\otimes\varpi_{k,i}.
\end{align}
We now add similar perturbations to the map $u$ as in Proposition \ref{p:stage} in order to remove most of the metric error. This can be done as in Proposition 3.1 in \cite{CS19}, which we can directly apply since  from property (c) and \eqref{e:bk} we deduce
$$
\|\chi_q^{i}b_{k,i}\|_{C^j(\Sigma_q^{i})}\leq Cd_q^{1/2-j}\,.
$$
Thus the assumptions of Proposition 3.1 in \cite{CS19} hold in each $\Sigma_q^{i}$ with parameters
$$
\delta=d_q,\quad\varepsilon=d_q,\quad \theta=d_q^{-1}, \quad \tilde\theta=d_q^{-1}.
$$
  Observe that, using property (a), we may ``add'' each primitive metric  $(\chi_q^{i}b_{k,i})^2\varpi_{k,i}\otimes\varpi_{k,i}$ with $q$ odd in parallel, and serially\footnote{In fact, one could also expoit the fact that the codimension $m-n\geq 2n_*$ to perform the steps in $k$ simultaneously as well. This would lead to an improved bound in \eqref{e:v-c2-estimate}, but this is not needed for our purpose.} in $i$ and $k$. We then repeat the same process for $q$ even. Proposition 3.1 in \cite{CS19} then yields, for any $K\geq C(M,\gamma)$,  an embedding $v\in C^2(\Sigma_\epsilon^{+},\R^{m})$ such that for all $q\in\N$ and $i=1,\ldots,N$
\begin{align}
\|v-u\|_{C^0(\Sigma_q^{i})}&\leq C(M, \gamma)d_q^{3/2}{K^{-1}},\label{e:v-c0-estimate}\\
 \|v-u\|_{C^1(\Sigma_q^{i})}&\leq C(M, \gamma)d_q^{1/2},\label{e:v-c1-estimate}\\
 \|v\|_{C^2(\Sigma_q^{i})}&\leq C(M, \gamma)d_q^{-1/2}K^{2N\tilde N}.\label{e:v-c2-estimate}
 \end{align}
 Since the perturbations in each step are compactly supported away from $\Sigma$ we have $u=v$  and  $du=dv$ along $\Sigma$.
Moreover

$$
v^{\sharp}e= u^{\sharp}e+\sum_{i=1}^{N}\sum_{k=1}^{\tilde N} \sum_{q=1}^{\infty}(\chi_q^{i}b_{k,i})^2\varpi_{k,i}\otimes\varpi_{k,i}+\mathcal{E}
$$
with
\begin{align*}
\|\mathcal{E}\|_{C^0(\Sigma_q^{i})}&\leq C(M,\gamma)d_q{K^{-1}},\\
\|\mathcal{E}\|_{C^1(\Sigma_q^{i})}&\leq C(M,\gamma)K^{2N\tilde N-1}\,,
\end{align*}
for every $i=1,\ldots,N$.
Now we are in a position to show that $v$ is our desired adapted short embedding. First of all, observe that for any $\theta_0<\frac12$ and any $i=1,\ldots, N$ by \eqref{e:v-c1-estimate}-\eqref{e:v-c2-estimate}
\begin{align*}
\|v-u\|_{C^{1, \theta_0}(\Sigma_q^{i})}&\leq \|v-u\|_{C^1(\Sigma_q^{i})}^{1-\theta_0}\|v-u\|_{C^2(\Sigma_q^{i})}^{\theta_0}\leq C(M,\gamma) d_q^{(1-2\theta_0)/2}
\end{align*}
is bounded independently of $q$ and $i$. Consequently $v\in C^{1,\theta_0}(\bar \Sigma_\epsilon^{+})$. Besides, for $(x,t)\in \Sigma_q^i$ we have $t\sim d_q\sim \rho^2(x,t)$. Therefore from \eqref{e:adapted-rho} and \eqref{e:v-c2-estimate}, we get
\begin{equation}\label{e:adapted-rho-v}
\begin{split}
&|\nabla\rho(x,t)|\leq C(M, \gamma)\rho(x,t)^{-1}\leq C(M, \gamma)\rho(x,t)^{1-\frac1{\theta_0}},\\
&|\nabla^2v(x,t)|\leq C(M, \gamma)\rho(x,t)^{-1}\leq C(M, \gamma)\rho(x,t)^{1-\frac1{\theta_0}},\\
\end{split}
\end{equation}
Similarly,
\begin{align*}
|\mathcal{E}(x,t)|&\leq C(M,\gamma)K^{-1}\rho^{2}(x,t)\,,\\
|\nabla\mathcal{E}(x,t)|&\leq C(M,\gamma)K^{2N\tilde N-1}.
\end{align*}
Let
$$
h=-\frac{\mathcal{E}}{\alpha\rho^2},
$$
so that
\begin{equation*}
g-v^{\sharp }e=\alpha\rho^2g-\mathcal{E}=\alpha\rho^2(g+h),
\end{equation*}
and then
\begin{equation}\label{e:adpated-h}
\begin{split}
|h(x,t)|&\leq C(M,\gamma)(\alpha K)^{-1}<\frac{\sigma_0}{4^{n+1}},\\
|\nabla h(x,t)|&\leq C(M,K,\gamma)(\alpha^{1/2}\rho(x,t))^{-2}+ C(M,\gamma)(\alpha K)^{-1}\rho(x,t)^{-\frac{1}{\theta_0}} \\
&\leq C(M,K, \gamma)(\alpha^{1/2}\rho(x,t))^{-\frac1{\theta_0}},
\end{split}
\end{equation}
provided  $K$ is taken large enough depending on $M, \gamma, \alpha,\sigma_0.$

{\bf Step 3. Isometric extension.} By the construction of $v$, we therefore have
$$g- v^{\sharp}e=\alpha\rho^2(g+h)$$
on $\Sigma_\epsilon^{+}$, $v$ is isometric on $\Sigma$ and additionally $v=u$, $dv=du$ on $\Sigma$. Thus in particular $dv(\nu) = du(\nu) = \mu$ along $\Sigma$ and therefore \eqref{e:flexibility} holds with $v$ replacing $u$.  Besides, we have $\alpha^{1/2}\rho \leq \frac{1}{4}$ and
\begin{align*}
|\nabla(\alpha^{1/2}\rho)|\leq A(\alpha^{1/2}\rho)^{1-\frac{1}{\theta_0}},\quad &|\nabla^2v|\leq A(\alpha^{1/2}\rho)^{1-\frac{1}{\theta_0}},\\
|h|\leq \frac{\sigma_0}{4^{n+1}},\quad &|\nabla h|\leq A(\alpha^{1/2}\rho)^{-\frac1{\theta_0}}.
\end{align*}
Now fix a triangulation of $\Sigma$ by $(n-1)$-simplices, such that every simplex is contained in a single chart $V_i$. Given any $(n-1)$-simplex $\Delta^{n-1}$, we can subdivide the product $\Delta^{n-1}\times[0,\epsilon]$ in a standard way (see for example \cite{Hatcher}) into a number of $n$-simplices. We then use the map $F$ from \eqref{d:F} to obtain a triangulation $\mathcal{T}$ of $\bar \Sigma_\epsilon^{+}$.

 Now set $S= \Sigma \cup \mathcal V$, where $\mathcal V$ is the vertex set of the triangulation. Then $\Sigma$ and $S$ satisfy Condition \ref{c:geometric}, and  therefore  we can apply Proposition \ref{p:inductive} to obtain a new adapted short embedding with respect to $S$.  Iterating the construction, i.e., setting $S_0= S$, $S_k= \Sigma\cup  T_k$, where $T_k$ is the union of the $k$-faces of the triangulation, and $\Sigma_k = S_{k-1}$ for $k=1,\ldots,n$ we finally end up with a adapted  short embedding with respect to $S_n = \bar \Sigma_\epsilon^{+}$, i.e., $\bar u :\bar \Sigma_\epsilon^{+}\to\R^{m}$ is an isometric embedding. It holds $\bar u\in C^{1,\theta'}\left (\bar \Sigma_{\epsilon}^{+},\R^{m}\right )$ for
 \[ \theta' = \theta_0 b^{-2n}\,,\]
 where $b>1$ is arbitrary. Since $\theta_0< \frac{1}{2}$ is arbitrary as well, it follows that we can achieve any regularity $C^{1,\theta }$ for $\theta<\frac{1}{2}$.

Finally, $\bar u=v=f$ on $\Sigma$, so $\bar u$ extends $f$. Moreover $d\bar u=dv$ on $\Sigma$, so that also
\[ \langle d\bar  u(\nu),\bar L(X,X)\rangle = \langle dv(\nu),\bar L(X,X)\rangle >  L(X,X) \,\]
for any tangent vector $X$ to $\Sigma$, finishing the proof.
%Therefore,  for any $\theta<\frac{1}{2},$ we can seek $b>1$ such that $\theta_0=b^2\theta<\frac{1}{2},$ and then $\bar u$ is our desired isometric extension of $\Sigma$ from $u$.
%
%
%
%

%%%%%%%%%%%%%%%%%%
\section{Proof of Theorem \ref{t:global}}\label{s:global}
%%%%%%%%%%%%%%%%%

We will concentrate on the case of immersions. The extension to embeddings is straight-forward and follows well-established strategies (see \cite{Nash54,SzLecturenotes,DIS}).

With Proposition \ref{p:inductive} at our disposal, the strategy for proving Theorem \ref{t:global}  for immersions is clear: we perform an induction on dimension on the skeleta of a given regular triangulation of $\mathcal{M}$.

As in Section \ref{s:iteration} we fix a finite atlas of charts $\{\Omega_k\}$ on $\mathcal{M}$ such that on every chart ${\gamma}^{-1}\mathrm{Id}\leq G\leq \gamma\mathrm{Id}$ and $\textrm{osc}_{\Omega_k} G\leq \sigma_0(\gamma)/2$ for some $\gamma>1$, where $\sigma_0(\gamma)$ is the constant given in Proposition \ref{p:stage}. In addition, fix a triangulation $\mathcal{T}$ on $\mathcal{M}$ whose skeleta consist of a finite union of $C^1$ submanifolds, such that each triangle $T\in\mathcal{T}$ is contained in a single chart.

We first take any $C^\infty$ embedding of $\mathcal{M}$ in $\R^{n+2n_*}.$ Then we change a scale of such embedding such that the resulting immersion, which we denote by $u$, is short.  By compactness of $\mathcal{M}$  we may also make $u$ strictly short, i.e.
\[ g- u^{\sharp}e >0 \]
on $\mathcal M$ in the sense of quadratic forms. Next, we will start our inductive construction as in \cite{CS20}. In the firt step, we recall the construction of an adapted short immersion $\tilde u$ of $\mathcal{M}$ with respect to $\Sigma=\emptyset$.

\begin{proposition}\label{p:strong}
Let $u\in C^2(\mathcal{M};\R^{n+2n_*})$ be a strictly short immersion. There exists $0<\delta^*\leq 1/8$ and $A^*\geq 1$, depending on $u$ and $g$, such that for any $A\geq A^*$ there exists a strictly short immersion $\tilde{u}$ and associated $\tilde{h}$ with
\begin{equation}\label{e:strong-1}
g-\tilde{u}^\sharp e=\delta^*(g+\tilde{h}),
\end{equation}
with
\begin{equation}\label{e:strong-2}
\tfrac{1}{2}g\leq \tilde{u}^\sharp e\leq g
\end{equation}
and such that the following estimates hold:
\begin{align}
&\|\tilde{u}-u\|_0\leq \delta^*A^{-\alpha^*}, \quad \|\tilde{u}\|_2\leq A, \label{e:strong-u}\\
&\|\tilde{h}\|_0\leq A^{-\alpha^*} ,\quad \|\tilde{h}\|_1\leq A. \label{e:strong-h}
\end{align}
The exponent $\alpha^*$ only depends on $\mathcal{M}$.
\end{proposition}
Next, fix $\theta_0<1/2$ and $\epsilon>0$. Set $u_0=\tilde u$, $h_0=\tilde h$ as obtained from Proposition \ref{p:strong} with $A=A_0$ sufficiently large (to be determined below), and also $\tilde\rho^2=\delta^*$. From  \eqref{e:strong-u} we deduce
\begin{equation}\label{e:u-tilde-u}
\|u-u_0\|_0\leq\frac{\eps}{4}
\end{equation}
by assuming $A_0$ is sufficiently large. From \eqref{e:strong-u}-\eqref{e:strong-h}, we further have
\begin{align*}
\|\nabla^2u_0\|_0&\leq A_0\leq A_0(\delta^*)^{\frac{1}{2}-\frac{1}{2\theta_0}}\,,\\
\|h_0\|_0&\leq A_0^{-\alpha^*}\leq \frac{\sigma_0}{4^{n+1}},\\
\|\nabla h_0\|_0&\leq A_0^{1-\alpha^*}\leq A_0(\delta^*)^{\frac{\alpha_0}{2}-\frac{1}{2\theta_0}},
\end{align*}
where $\sigma_0$ is in Proposition \ref{p:stage}. Therefore we deduce that $u_0$ is an adapted short immersion with respect to the empty set $\Sigma_0=\emptyset$ with exponent $\theta_0$, and furthermore the estimates \eqref{e:inductive-u-rho-assumption} are satisfied by $(u_0, \rho_0, h_0)$ with $(A,\theta)$ replaced by $(A_0,\theta_0)$.

For any $b>1$, we can apply Proposition \ref{p:inductive} to obtain a $C^{1, \theta_1}$ adapted short immersion $(u_1, \rho_1, h_1)$ with respect to $\Sigma_1=\mathcal{V}$, where $\mathcal{V}$ is the vertex set of the triangulation $\mathcal{T}$ and such that \eqref{e:inductive-u-rho-assumption}\eqref{e:inductive-conclusion} hold with
$$
A_1=A_0^{b^2}, \quad \theta_1=\frac{\theta_0}{b^2}.
$$
We then continue this process along the skeleta $\Sigma_1\subset\Sigma_2\subset\dots\Sigma_{n+1}=\mathcal{M}$ and obtain
adapted short immersions $(u_j, \rho_j, h_j)$ with respect to $\Sigma_j$, $j=1,2,\dots,n+1$
with
$$
A_{j+1}=A_j^{b^2}, \quad \theta_{j+1}=\frac{\theta_j}{b^2}.
$$
After $n+1$ steps we finally obtain a global $C^{\theta_{n+1}}$ isometric immersion $v:=u_{n+1}$ of $\mathcal{M}$, with
$$
\theta_{n+1}=b^{-2n-2}\theta_0.
$$
Note that for any fixed $\theta_0$, taking $b\to1,$ we will have $\theta_{n+1}\to \theta_0$. Thus, for any $\theta'<\theta_0$ there exists a choice of $b>1$ so that $\theta'<\theta_{n+1}<\theta_0$. In this way we can achieve any exponent $\theta<\frac12$.
Finally, observe that (recalling \eqref{e:u-tilde-u})
\begin{align*}
\|u-v\|_0&\leq\|u-u_0\|_0+\sum_{j=0}^n\|u_{j+1}-u_j\|_0\\
&\leq \eps/4+\sum_{j=0}^nA_j^{-1/2}\leq \eps/4+(n+1)A_0^{-1/2}\\
&\leq \eps
\end{align*}
by choosing $A_0$ sufficiently large. This completes the proof of Theorem \ref{t:global}.

%%%%%%%%%%
%
\bigskip

\begin{appendix}
\section{Proof of Lemma \ref{l:normal}}

\textbf{Step 1.} Without loss of generality we assume $\Omega = B_1(0)$. In a first step we construct a family of vectorfields $\zeta_1,\ldots, \zeta_{m-n}$ which satisfies \eqref{e:normalestimates} on a small neighbourhood of the origin. To do this, pick orthonormal vectors $\xi_1,\ldots, \xi_{m-n} \in \R^{m}\setminus dv_0\left (T_0\bar B_1\right )$. We then set
 \[ \nu_i = \xi_i -\sum_{j=1}^{n} r_{ij}\partial_jv\,,\]
where $r_{ij}$ are chosen to guarantee $\langle \nu_i,\partial_kv\rangle =0$ for every $i$ and $k$. This is possible since $\nabla v ^{T} \nabla v \geq \gamma^{-1} Id$. Indeed, denote $b_{ik} = \langle \xi_i,\partial_k v\rangle $ and observe that $\langle \nu_i,\partial_kv\rangle = 0$ for all $i,k$ is equivalent to
\[ R\cdot \nabla v^{\intercal} \nabla v = B\,,\]
where $R$ and $B$ are the $(m-n)\times n$ matrices with entries $R_{ij}=r_{ij}$ and $B_{ij} = b_{ij}$. We can then simply set
\[ R = B\cdot \left (\nabla v^{\intercal }\nabla v\right )^{-1 }\,.\]
We claim that in a neighbourhood of the origin the family $\{ \nu_i\}_{i=1}^{m-n}$ is linearly independent and therefore constitutes a frame for the normal bundle. A Gram-Schmidt process will then produce the desired vectorfields.

To show the claim, we write
\[ (\nabla v^{\intercal } \nabla v)^{-1}_{ij} = (\det \nabla v^{\intercal }\nabla v)^{-1} P_{ij}(\nabla v)\,,\]
where $P_{ij}(\nabla v)$ is a polynomial in the arguments $\partial_k v^{l}$. Observe that assumption \eqref{e:vispositivedefinite} implies $[v]_1 \leq C(\gamma)$. Hence, with Lemma \ref{l:composition} and assumption \eqref{e:vispositivedefinite} we find
%\[ [ P_{ij}(\nabla v) ]_{l,\beta}\leq C_{l,\beta}(\gamma)[v]_{l+1,\beta}\,.\]
%Moreover, again thanks to \eqref{e:vispositivedefinite},
%\[ [(\det\nabla v^{\intercal }\nabla v)^{-1}]_{l,\beta}\leq C_{l,\beta}(\gamma) [v]_{l+1,\beta}\,,\]
%so that
%\begin{equation} \label{e:p4}
% [(\nabla v^{\intercal } \nabla v)^{-1}_{ij}]_{l,\beta}\leq C_{l,\beta} [v]_{l+1,\beta}\,.
%\end{equation}
%In particular
%\[ \|(\nabla v^{\intercal } \nabla v)^{-1}_{ij}\|_0\leq C(\gamma)\,.\]

\[ \|r_{ij}\|_0 \leq C(\gamma)[v]_{2} \varepsilon\,\]
where we used that
\[  |b_{ik}|=| \langle \xi_i,\partial_k v(x)\rangle | = | \langle \xi_i,\partial_k v(x)-\partial_kv(0)\rangle |\leq [v]_{2} \varepsilon\,\]
for $x\in \bar B_\varepsilon$.
With this estimate we find
\[ \| \langle \nu_i,\nu_j\rangle -\delta_{ij}\|_{C^{0}(\bar B_\varepsilon)} = \| \langle \nu_i,\nu_j\rangle -\langle \xi_i,\xi_j\rangle \|_{C^{0}(\bar B_\varepsilon)}\leq C(n,\gamma)[v]_{2}\varepsilon\,.\]
Therefore, if $\varepsilon\equiv \varepsilon\left (n,\gamma,[v]_{2}\right )>0$ is small enough, the vectorfields $\nu_1,\ldots, \nu_{m-n}$ are linearly independent. Before continuing with the Gram-Schmidt process, observe the following estimates for $0<l\leq N$
\[ [r_{ij}]_{l} \leq C_{l}\left (C(\gamma) [v]_{l+1} + C(\gamma) [b_{ij}]_{l} \right )\leq C_{l}(\gamma)[v]_{l+1}\,,\]
thanks to the Leibniz rule. Therefore we have the same estimates for the vectorfieds
\begin{equation}\label{e:nuestimates}
 [\nu_i]_{l} \leq C_{l}(\gamma)[v]_{l+1}\,.
\end{equation}
Now we set
\[ \zeta_1 = \frac{\nu_1}{|\nu_1|}\,,\]
and observe that for small enough $\varepsilon>0$ we have $|\nu_1|\geq \frac{1}{2}$ so that, thanks to Lemma \ref{l:composition} and \eqref{e:nuestimates}, $ \zeta_1\in C^{N}\left (\bar B_\varepsilon\right )$ with
\[ [\zeta_1]_{C^{l}(\bar B_\varepsilon)}\leq C_{l}[\nu_1]_{C^{l}(\bar B_\varepsilon)} \leq C_{l}(\gamma)[v]_{C^{l}(\bar B_\varepsilon)}\,\]
for all $0\leq l \leq N$. Moreover, on $\bar B_\varepsilon $ we have
\[ | \zeta_1- \xi_1| \leq \frac{2|\nu_1-\xi_1|}{|\nu_1|}\leq C(\gamma)[v]_{2} \varepsilon\,.\]
Now we assume $\zeta_1,\ldots, \zeta_{k-1}$ are already constructed with
\begin{align*}
 \langle \zeta_i,\zeta_j\rangle &= \delta_{ij} \\
  \nabla v\cdot \zeta_i &= 0 \\
  [\zeta_i]_{l,\beta} &\leq C_{l}(\gamma)[v]_{l+1}  \quad \text{ for all } 0\leq l \leq N\,,
\end{align*}
on $\bar B_\varepsilon$, and in addition
\begin{equation}\label{e:p6}
|\zeta_i-\xi_i|\leq C(\gamma)[v]_{2} \varepsilon\,.
\end{equation}
We then set
\[ \theta_k = \nu_k - \sum_{j=1}^{k-1}\langle \nu_k,\zeta_j\rangle \zeta_j\,, \quad \zeta_k = \frac{\theta_k}{|\theta_k|}\,.\]
It remains to show that $\zeta_k$ satisfies \eqref{e:normalestimates} and \eqref{e:p6}. Observe that
 \[ \langle \nu_k, \zeta_j\rangle  = \langle \nu_k-\xi_k,\zeta_j\rangle+ \langle \xi_k,\zeta_j-\xi_j\rangle \,\]
 so that $|\langle \nu_k,\zeta_j\rangle | \leq C(\gamma)[v]_{2}\varepsilon$ on $\bar B_\varepsilon$  and by the Leibnitz rule also   \[
 [ \langle \nu_k,\zeta_j\rangle]_{C^{l}(\bar B_\varepsilon)} \leq C_l(\gamma)[v]_{C^{l+1}(\bar B_\varepsilon)} \,.\]
 In particular $|\theta_k|\geq \frac{1}{4}$ for $\varepsilon$ small enough and hence, with Lemma \ref{l:composition},
 \[ [\zeta_k]_{C^{l}(\bar B_\varepsilon)}\leq C_l(\gamma)[v]_{C^{l+1}(\bar B_\varepsilon)}\,.\]
 Therefore $\zeta_k$ satisfies \eqref{e:normalestimates}. Since moreover
\[ |\zeta_k- \xi_k |\leq \frac{2|\theta_k-\xi_k|}{|\theta_k|} \leq C( |\theta_k-\nu_k|+|\nu_k-\xi_k|) \leq C(\gamma)[v]_{2}\varepsilon\,,\]
the first step is completed.
\medskip

\noindent \textbf{Step 2.} In this step we show that one can continue the vectorfields to maps on $\bar B_1$ satisfying the same constraints. Consider the set
\[ R = \{ \rho \in [0,1]: \text{there exist } \zeta_1,\ldots, \zeta_{m-n} \in C^{N,\alpha}(\bar B_\rho) \text{ satisfying \eqref{e:normalestimates} on } \bar B_\rho\}\,.\]
As we saw in Step 1, $R$ is nonempty. Set $\bar \rho = \sup R$. We claim that $\bar \rho \in R$. To see this, let $\rho_q\uparrow\bar \rho$ and fix the corresponding families of vectorfields $\zeta_i^{q}$. Now assume that there exists $\delta =\delta(\gamma,v)> 0$ such that each $\zeta_i^{q}$ can be extended to a map $\tilde \zeta_i^{q}\in C^{N}(\bar B_{\sigma_q})$ with
\begin{equation}\label{e:p7}
[\,\tilde \zeta_i^{q}\,]_{C^{l}(\bar B_{\sigma_q})} \leq C(\gamma)(1+[v]_{l+1})\,,
\end{equation}
where $\sigma_q = \min\{1,\rho_q+\delta\}$. We will prove this fact at the end of this proof in Step 3. With it, we can repeat the procedure of Step 1: We set
\[ \nu_i^{q} = \tilde \zeta_i^{q}- \sum_{j=1}^{n}r_{ij}^{q}\partial_jv\,,\]
 where, again, $r_{ij}^{q}$ are chosen such that every $\nu_i^{q}$ is orthogonal to $v$. We need to show that, for $\delta$ small enough, $\nu_i^{q}$ are linearly independent to perform the Gram-Schmidt process. Set $b_{ik}^{q} = \langle \tilde \zeta_i^{q},\partial_k v\rangle$ and observe that, for $ \rho_q < |x| \leq \sigma_q $,
 \[ b_{ik}^{q} (x) = \Big\langle \tilde \zeta_i^{q}(x)-\zeta_i^{q}\left (\rho_q\frac{x}{|x|}\right ),\partial_k v(x)\Big\rangle + \Big\langle \zeta_i^{q}\left (\rho_q\frac{x}{|x|}\right ), \partial_kv(x)-\partial_kv\left (\rho_k\frac{x}{|x|}\right )\Big\rangle\,. \]
 Thus,
 \[|b_{ik}^{q}| \leq C(\gamma)[\,\tilde \zeta_i^{q}\,]_{C^{1}(\bar B_{\sigma_q})}\delta + [v]_{2}\delta\leq C(\gamma,[v]_{2})\delta\,\]
 thanks to \eqref{e:p7}. Thus, as before it follows
 \[ |r_{ij}^{q}|\leq C(\gamma,[v]_{2}) \delta\,.\]
 Now we write
 \[ \langle \nu_i^{q},\nu_j^{q} \rangle = \langle \tilde \zeta_i^{q},\tilde \zeta_j^{q}\rangle + E\,,\]
 where $E$ is an error term with  $|E|\leq C(\gamma,[v]_{2}) \delta$ thanks to the estimate on $r_{ij}^{q}$. We expand
 \begin{align*} \langle \tilde \zeta_i^{q},\tilde \zeta_j^{q}\rangle &= \Big\langle \tilde \zeta_i^{q}-\zeta_i^{q}\left (\rho_q \frac{x}{|x|}\right ),\tilde \zeta_j^{q}\Big\rangle+\Big\langle \tilde \zeta_i^{q}\left (\rho_q \frac{x}{|x|}\right ),\tilde \zeta_j^{q}-\zeta_j^{q}\left (\rho_q \frac{x}{|x|}\right )\Big\rangle \\
 & \quad + \Big\langle \zeta_i^{q}\left (\rho_q\frac{x}{ |x|} \right ),\zeta_j^{q}\left (\rho_q\frac{x}{ |x|} \right )\Big\rangle \\
 & = \delta_{ij} + \tilde E\,,
 \end{align*}
 where again $|\tilde E| \leq C(\gamma,[v]_{2}) \delta$.  Hence, for $\delta(\gamma,v)$ small enough, $\nu_i^{q}$ are linearly independent. The estimates \eqref{e:nuestimates} can be derived in the same way. As in Step 1, we can then apply the Gram-Schmidt process to generate the vectorfields $\bar \zeta_i^{q}$ satisfying \eqref{e:normalestimates} on $\bar B_{\sigma_q}$. Consequently, $\sigma_q \in R$. By definition, $\bar \rho \geq \sigma_q$ for all $q$. Letting $q\to \infty$ we find $\rho\geq \min\{ 1,\rho+\delta\}$, which shows $\sigma_q =1$ for $q$ large enough. Hence $1\in R$, which completes Step 2.
\medskip

\noindent \textbf{Step 3.} In this step we show that there exists a $\delta\equiv \delta(\gamma,v) > 0$ such that any map $\zeta \in C^{N}(\bar B_\rho )$ with
\begin{equation}\label{e:extensionassumption}
[ \zeta]_{C^{l}(\bar B_\rho)} \leq C_l(\gamma) \left (1+ [v]_{C^{l+1}(\bar B_1)}\right ) \,
\end{equation}
can be extended to a map $\tilde \zeta \in C^{N}(\R^{n})$ such that
\begin{equation}\label{e:p8}[\tilde \zeta ]_{C^{l}(\bar B_\sigma) } \leq C_l(\gamma)\left (1+[v]_{C^{l+1}(\bar B_1)}\right )\,,\end{equation}
where $\sigma = \min\{1,\rho+\delta\}$ and $C_l(\gamma)$ might differ from the constant in \eqref{e:extensionassumption}. 

The existence of such an extension is a classical fact, originally due to Whitney \cite{Whitney}. However, we could not find a reference stating the estimates \eqref{e:p8}, which is why we redo the argument in the following.

For $k\in \N$, $y\in \bar B_\rho$ and $x\in \R^{n}$ we denote by $T_y^{k}\zeta(x) $ the k-th order Taylor polynomial of $\zeta$ around $y$ at $x$, i.e.
\[ T_y^{k}\zeta(x) = \sum_{|\beta|\leq k} \frac{\partial^{\beta}\zeta(y)}{\beta!}(x-y)^{\beta}\,,\]
with  the usual conventions concerning the multi-indices $\beta$. Let $\chi_j$ be a partition of unity subordinate to decomposition of $\R^{n}\setminus \bar B_\rho$ such that no point is in the support of more than $M(n)$ functions $\chi_j$, the diameter of the support of $\chi_j$ is at most twice its distance to $\bar B_\rho$, and
\[ |\partial^{\beta}\chi_j(x)| \leq C_\beta d(x)^{-|\beta|}\]
for $x\in \R^{n} \setminus \bar B_\rho$, where $d(x)= \mathrm{dist}(x,\bar B_\rho)$. For a proof we refer to Lemma 2.3.7. in \cite{Hoermander}.

 We then set $\tilde \zeta (x) = \zeta(x)$ for $x\in \bar B_\rho$ and
 \[ \tilde \zeta(x) = \sum_j \chi_j(x) T_{y_j}^{N}\zeta(x)\,\]
 otherwise, where $y_j \in \partial \bar B_\rho$ minimizes the distance to the support of $\chi_j$. In Theorem 2.3.6. in \cite{Hoermander} it is shown that $\tilde \zeta \in C^{N}$ with $\partial^{\beta} \tilde \zeta = \partial^{\beta}\zeta $ on $\bar B_\rho$ for every $|\beta| \leq k$. We want to show that $\tilde \zeta$ also satisfies \eqref{e:p8}.

  Observe first that if $x\in \supp \chi_j$ then
 \[ |x-y_j| \leq \mathrm{diam}\left (\supp\chi_j\right ) + \mathrm{dist}(\supp\chi_j,\bar B_\rho) \leq 3d(x)\,.\]
Hence, for such $x$ we have
\begin{align*} |\partial^{\beta}T^{N}_{y_j}\zeta(x)| &= \Big|\sum_{|\mu|\leq N-|\beta|} \frac{\partial^{\beta+\mu}\zeta(y_j)}{\mu!}(x-y_j)^{\mu}\Big| \leq [\zeta]_{|\beta|} + \sum_{i=1}^{N-|\beta|}d(x)^{i}[\zeta]_{|\beta|+i} \nonumber \\
&\leq [\zeta]_{|\beta|} + C_{N}(\gamma	)d(x)\left (1+\|v\|_{N+1}\right )\,,
\end{align*}
for any multi-index $\beta$ with $0\leq|\beta|\leq N$.
Consequently, if $ \rho < |x|\leq \rho+\delta$ we find
\begin{equation}\label{e:p9}
 |\partial^{\beta}T_{y_j}^{N}\zeta(x)| \leq [\zeta]_{|\beta|} +1\,,
\end{equation}
if $\delta$ is chosen small enough depending on $\gamma$ and $v$. In particular, this shows the estimate \eqref{e:p8} for $l=0$ in view of \eqref{e:extensionassumption}. Now fix  $1\leq l\leq N$ and multi-indices $\alpha,\beta$ with $|\alpha|+|\beta| = l$. We want to show the estimate
\begin{equation}\label{e:p10} \Big|\sum_j \partial^{\alpha}\chi_j\partial^{\beta}
T^{N}_{y_j}\zeta \Big| \leq C_{l}(\gamma)\left (1+[v]_{l+1}\right )\,.\end{equation}
If $\alpha=0$ this follows from the estimate \eqref{e:p9} together with the assumption \eqref{e:extensionassumption}. Therefore we can assume $|\alpha|\geq 1$. We write
\begin{align*}
\sum_j \partial^{\alpha}\chi_j(x)\partial^{\beta}
T^{N}_{y_j}\zeta(x) &= \sum_j \partial^{\alpha}\chi_j(x)\sum_{|\mu|\leq N-|\beta| } \frac{\partial^{\beta+\mu}\zeta(y_j)}{\mu!}(x-y_j)^{\mu} \\ &=    \sum_j \partial^{\alpha}\chi_j(x)\sum_{|\mu|<|\alpha| } \frac{\partial^{\beta+\mu}\zeta(y_j)}{\mu!}(x-y_j)^{\mu}\\
&\quad + \sum_j \partial^{\alpha}\chi_j(x)\sum_{|\alpha|\leq |\mu|\leq N-|\beta| } \frac{\partial^{\beta+\mu}\zeta(y_j)}{\mu!}(x-y_j)^{\mu}\\
& =: \mathrm{I}(x)+\mathrm{II}(x)
\end{align*}
Recall that $|\partial^{\alpha}\chi_j(x)|\leq Cd(x)^{-|\alpha|}$. Since $|\alpha|=l-|\beta|$ we can estimate the second sum by
\begin{align*}|\mathrm{II}(x)|&\leq Cd(x)^{-|\alpha|}\left ( [\zeta]_l d(x)^{|\alpha|} + \sum_{i=1}^{N-|\beta|} d(x)^{|\alpha|+i}[\zeta]_{l+i}\right ) \\
&\leq C[\zeta]_l +C_{N}(\gamma)d(x)\left (1+\|v\|_{N+1}\right )\leq C_l(\gamma)(1+[v]_{l+1})\,
\end{align*}
if $\rho< |x|\leq \rho+\delta$ and $\delta$ small enough, i.e. $d(x)\leq\delta,$ thanks to \eqref{e:extensionassumption}. To estimate $\mathrm{I}(x)$ we set $x^{*}= \rho \frac{x}{|x|}$ and observe that, by Taylor's theorem,
\[ \partial^{\beta+\mu}\zeta(y_j)-T^{|\alpha|-|\mu|-1 }_{x*}\partial^{\beta+\mu}\zeta(y_j) = \sum_{|\tilde\mu|=|\alpha|-|\mu|} \frac{\partial^{\tilde\mu+\beta+\mu}\zeta(\xi)}{\tilde\mu!}(y_j-x^{*})^{\tilde\mu} \,,\]
for some $\xi\in [x^{*},y_j]$. Now $|x^{*}-y_j|\leq d(x)+|x-y_j|\leq 4d(x)$, so that
\[\Big|\partial^{\beta+\mu}\zeta(y_j)-T^{|\alpha|-|\mu|-1 }_{x*}\partial^{\beta+\mu}\zeta(y_j)\Big|\leq C[\zeta]_{|\alpha|+|\beta|} d(x)^{|\alpha|-|\mu|}= C[\zeta]_{l} d(x)^{|\alpha|-|\mu|}\,.\]
 since by assumption $|\alpha|+|\beta|=l$. Therefore,  it holds
\begin{equation}\label{e:finales} \Big|\sum_j \partial^{\alpha}\chi_j(x)\sum_{|\mu|<|\alpha|}\frac{(x-y_j)^{\mu}}{\mu!}\left (\partial^{\beta+\mu} \zeta(y_j) -T^{|\alpha|-|\mu|-1 }_{x^{*}}\partial^{\beta+\mu}\zeta(y_j)\right )\Big|\leq C[\zeta]_l\,.
\end{equation}
To conclude it suffices to observe
\begin{align}
&\sum_{|\mu|\leq |\alpha|-1}\frac{(x-y_j)^{\mu}}{\mu!}T_{x^{*}}^{|\alpha|-|\mu|-1}\partial^{\beta+\mu}\zeta(y_j) \nonumber\\
&\qquad \quad \quad = \sum_{|\mu|\leq |\alpha|-1}\left (\sum_{|\tilde\mu|\leq |\alpha|-|\mu|-1}\frac{\partial^{\tilde\mu+\beta+\mu}\zeta(x^{*})}{\tilde\mu!\mu!}(x-y_j)^{\mu}(y_j-x^{*})^{\tilde\mu}\right )\nonumber\\
& \qquad \quad \quad = \sum_{|\mu|\leq |\alpha|-1} \frac{\partial^{\beta+\mu}\zeta(x^{*})}{\mu!}\left (\sum_{\tilde\mu \leq \mu } \frac{\mu!}{\tilde\mu!(\mu-\tilde\mu)!}(x-y_j)^{\mu}(y_j-x^{*})^{\mu-\tilde\mu}\right )\nonumber\\
&\qquad \quad \quad   = \sum_{|\mu|\leq |\alpha|-1} \frac{\partial^{\beta+ \mu}\zeta(x^{*})}{\mu!}(x-x^{*})^{\mu} = T^{|\alpha|-1}_{x^{*}}\partial^{\beta}\zeta(x)\,. \label{e:p11}
\end{align}
Since $\sum_j \partial^{\alpha}\chi_j(x) = 0$ we can simply subtract $T^{|\alpha|-1}_{x^{*}}\partial^{\beta}\zeta(x)$ to find
\[ |\mathrm{I}(x)| = \Big|\sum_j\partial^{\alpha}\chi_j(x)\left (\sum_{|\mu|\leq |\alpha|-1} \frac{\partial^{\beta+\mu}\zeta(y_j)}{\mu!}(x-y_j)^{\mu} - T^{|\alpha|-1}_{x^{*}}\partial^{\beta}\zeta(x)\right )\Big|\leq C[\zeta]_l\]
in view of \eqref{e:p11} and \eqref{e:finales}, which, thanks to \eqref{e:extensionassumption}, finishes the proof.
\end{appendix}

%\begin{acknowledgements}
%\end{acknowledgements}



\begin{thebibliography}{99}

\bibitem{Borisov58}
Yu. F. Borisov, The parallel translation on a smooth surface. {I-IV}.
{\em Vestnik Leningrad. Univ. 13,14}, (1958,1959).

\bibitem{BorisovRigidity1}
Yu. F. Borisov, On the connection between the spatial form of smooth surfaces and
  their intrinsic geometry. {\em Vestnik Leningrad. Univ. 14}, 13 (1959), 20--26.

\bibitem{Bor65}
Y.F. Borisov,  $C^{1, \theta}$-isometric immersions of Riemannian spaces. (Russian) {\it Dokl. Akad. Nauk SSSR} { 163} (1965), 11 C-3.

\bibitem{Bor04}
Y.F.  Borisov,  Irregular surfaces of the class $C^{1, \beta}$ with an analytic metric. (Russian) {\it Sibirsk. Mat. Zh.} {45} (2004), no. 1, 25-61; {\it translation in Siberian Math. J. } { 45} (2004), no. 1, 19-52


\bibitem{BDSV17}
T. Buckmaster, C. De Lellis, L. Sz\'ekelyhidi Jr., and V. Vicol, Onsager's conjecture for admissible weak solutions {\it Comm. Pure Appl. Math.} { 72} (2019), no. 2, 229-274.


\bibitem{Cao}
W. Cao, The semi-global isometric embedding of surfaces with curvature changing signs stably. {\it Proc. Amer. Math. Soc.} 147 (2019), no. 10, 4343-4353.


\bibitem{CS19}
W. Cao, L. Sz\'ekelyhidi Jr., $C^{1, \alpha}$ isometric extension, {\it  Comm. Partial Differential Equations}, (2019), no. 7, 613-636.

\bibitem{CS20}
W. Cao, L. Sz\'ekelyhidi Jr., Global Nash-Kuiper theorem for compact manifolds. {\it To appear in J. Diff. Gerometry. ArXiv preprint:1906.08608v2 (2019).}


\bibitem{CDS12}
S. Conti, C. De Lellis, and L. Sz\'ekelyhidi Jr. $h$-principle and rigidity for $C^{1,\alpha}$ isometric embeddings. In
{\it Nonlinear paritial differential eqautions. The Abel aymmposium 2010. Proceedings of the Abel symposium, Oslo, Norway, September 28-October 2, 2010,} 83-116, Berlin Springer, 2012.

\bibitem{CET}
P. Constantin, W. E, and E.S. Titi,
Onsager's conjecture on the energy conservation for solutions of Euler's equation.
{\it Comm. Math. Phys.} 165 (1994), no. 1, 207-209.

\bibitem{CohnVossen}
S. Cohn-Vossen, Zwei S\"atze \"uber die Starrheit der Eifl\"achen. {\it Nachrichten Ges.~d.~Wiss zu G\"ottingen} (1927), 125-134.



\bibitem{DIS}
C. De Lellis, D. Inauen, and L. Sz\'ekelyhidi Jr. A Nash-Kuiper theorem for $C^{1, 1/5-\delta}$ immersions on surfacesin 3 dimensions, {\it Rev. Mat. Iberoamericana} 34 (2018), 1119-1152.

\bibitem{DIS20}
C. De Lellis and D. Inauen, {{C$^{1,\alpha}$ Isometric embeddings of polar caps}},
{\it Adv. Math.} 363 (2020), 106996.



\bibitem{Dong}
G.-C. Dong, The semi-global isometric embedding in $\R^3$ of two-dimensional Riemannian manifolds with Gaussian curvature changing sign cleanly. {\it J. Partial Differential Equations}, { 6}(1993),62-79.


\bibitem{Gromov:2015tua}
M. Gromov, {{Geometric, algebraic, and analytic descendants of Nash isometric embedding theorems}}, {\it Bull. Amer. Math. Soc. (N.S.)} 54, { 2} (2017) 173--245.


\bibitem{Han05}
Q. Han,  On isometric embedding of surfaces with Gauss curvature changing
sign cleanly. {\it Comm. Pure Appl. Math.} {  58} (2005), 285-295.

\bibitem{Han06}
Q. Han, Local isometric embedding of surfaces with Gauss curvature changing sign stably across a curve.
{\it Calc. Var. Partial Differential Equations} {25} (2006), no. 1, 79 C103

\bibitem{Hatcher}
A. Hatcher, Algebraic topology, {\it Cambridge University Press}, Cambridge, UK (2002).

\bibitem{Herglotz}
 G. Herglotz,  {\"U}ber die Starrheit der Eifl{\"a}chen {\it Abhandlungen aus dem Mathematischen Seminar der  Universit{\"a}t Hamburg}, vol. 15, no. 1, pp. 127--129, (1943)
\bibitem{Hoermander}
L. H\"ormander, The analysis of linear partial differential operators. I. Distribution theory and Fourier analysis, Springer-Verlag, (1990)


\bibitem{HunWas16}
 N. Hungerb\"{u}hler and M. Wasem. The one-sided isometric extension problem.
 {\it Results Math.} { 71}(2017), no. 3-4, 749-781.



\bibitem{Ise16}
P. Isett, A Proof of Onsager's conjecture.
{\it Ann. of Math. (2)} {188} (2018), no. 3, 871-963.

 \bibitem{Jac74}
 H. Jacobowitz. Extending isometric embeddings. {\it J. Differential Geometry,} 9: 291-307, 1974.

 \bibitem{K78}
 A. K\"all\'en,  Isometric embedding of a smooth compact manifold with a metric of low regularity. {\it Ark. Mat.} {16} (1978), no. 1, 29-50.



 \bibitem{Kh}
M. Khuri, The local isometric embedding in $\R^3$ of two-dimensional Riemannian manifolds with Gaussian curvature changing sign to finite order on a curve.
{\it J. Differential Geom.} { 76} (2007),




 \bibitem{Kui55}
 N.H. Kuiper. On $C^1$ isometric embeddings. I, II. {\it Nederl. Akad. Wetensch. Proc. Ser. A. 58= Indag. Math.} 17: 545-556, 683-689, 1955.


\bibitem{Lin}
C.S. Lin, The local isometric embedding in $\R^3$ of 2-dimensional Riemannian
manifolds with Gaussian curvature changing sign cleanly.
{\it Comm. Pure Appl. Math.} { 39} (1986), 867-887.


\bibitem{Nash54} J. Nash. $C^1$ isometric embeddings. {\it Ann. of Math. (2)}, 60: 383-396, 1954.


\bibitem{PogorelovRigidity}
A. Pogorelov, The rigidity of general convex surfaces. {\em Doklady Acad. Nauk SSSR { 79}} (1951), 739--742.

\bibitem{SzLecturenotes}
L. Sz\'ekelyhidi, Jr.  From isometric embeddings to turbulence. {\it HCDTE lecture notes. Part II. Nonlinear hyperbolic PDEs, dispersive and transport equations,} 63 pp., AIMS Ser. Appl. Math., 7, Am. Inst. Math. Sci. (AIMS), Springfield, MO, 2013.

\bibitem{Whitney}

H. Whitney, Analytic extensions of functions defined in closed sets, {}{\it Transactions of the American Mathematical Society} 36 (1), 63-89, (1934)
\end{thebibliography}
\end{document}